\documentclass[12pt,a4paper,reqno]{amsart}

\usepackage{amsrefs}

\usepackage{calc,verbatim,enumitem,tikz,url,hyperref,mathrsfs,cite}

\usepackage{thm-restate}

\usepackage{float}
\usepackage{xcolor}
\usepackage{marginnote}
\usepackage{mathrsfs}
\usepackage{extarrows}
\usepackage{calligra}
\usepackage{tabularx}
\usepackage{booktabs}
\usepackage{pgfplots}

\pgfplotsset{
	standard/.style={
		axis x line=middle,
		axis y line=middle,
		every axis x label/.style={at={(current axis.right of origin)},anchor=north west},
		every axis y label/.style={at={(current axis.above origin)},anchor=north east}
	}
}

\pgfplotsset{width=13cm,compat=1.13}

\usepgfplotslibrary{fillbetween}

\usepackage{amsmath, amssymb, latexsym, amsthm, amsbsy}

\usepackage{graphics}

\usepackage{graphicx,relsize}

\usepackage[color, all]{xy}

\usepackage[mathscr]{eucal}

\usepackage{bbm}

\usepackage[a4paper,top=3cm,bottom=3cm,inner=3cm,outer=3cm,includehead,includefoot]{geometry}

\newcommand{\subgp}[1]{\langle{{\hash}1}\rangle}

\def\S{\mathcal{S}}

\def\x{\textbf{x}}
\def\y{\textbf{y}}

\def\XX{\mathbf{X}}

\def\Bern{\textrm{Bern}}

\def\Bin{\mathbf{Bin}}

\newcommand{\Ex}[1]{\mathbb{E}\left[#1\right]}

\newcommand{\pr}[1]{\mathbb{P}\left(#1\right)}
\newcommand{\expb}[1]{\exp\left(#1\right)}

\def\Pr{\mathbb{P}}

\newcommand{\eq}[1]{\begin{equation}\label{eq:#1}}
	\newcommand{\eqe}{\end{equation}}
\newcommand{\eqr}[1]{\eqref{eq:#1}}

\numberwithin{equation}{section}

\def\Var{\textup{Var}}

\def\RR{\mathbb{R}}

\newtheorem{theorem}{Theorem}[section]
\newtheorem{lem}[theorem]{Lemma}
\newtheorem{conjecture}[theorem]{Conjecture}
\newtheorem{cor}[theorem]{Corollary}
\newtheorem{claim}[theorem]{Claim}

\newtheorem{prop}[theorem]{Proposition}
\theoremstyle{definition}\newtheorem{definition}[theorem]{Definition}
\theoremstyle{definition}\newtheorem*{remark}{Remark}

\def\le{\leqslant}
\def\ge{\geqslant}

\renewcommand{\leq}{\leqslant}
\renewcommand{\geq}{\geqslant}

\def\eps{\varepsilon}

\usepackage{soul}

\definecolor{sgreen}{rgb}{0.3, 0.9, 0.3} 

\definecolor{lblue}{rgb}{0.6, 0.6, 1} 
\definecolor{lgr}{rgb}{0.8, 0.8, 0.8} 
\definecolor{purp}{rgb}{0.9, 0, 0.9} 

\definecolor{mgr}{rgb}{0.7, 0.7, 0.7} 
\definecolor{dmgr}{rgb}{0.6, 0.6, 0.6} 

\pgfplotsset{compat=1.13}
\usepgfplotslibrary{groupplots,fillbetween}

\begin{document}

	\title[Moderate deviations of triangle counts]{Moderate Deviations of Triangle Counts in the Erd\H{o}s-R\'enyi Random Graph $G(n,m)$: The Lower Tail}
	
	\author{Jos\'e D. Alvarado \and Gabriel Dias do Couto \and Simon Griffiths}

\address{Faculty of Mathematics and Physics, University of Ljubljana, Slovenia, and Institute of Mathematics, Physics and Mechanics, Slovenia.}\email{josealvarado.mat17@gmail.com}

\address{Departamento de Matem\'atica, PUC-Rio, Rua Marqu\^{e}s de S\~{a}o Vicente 225, G\'avea, 22451-900 Rio de Janeiro, Brazil}\email{gdiascouto@aluno.puc-rio.br, simon@puc-rio.br}

\thanks{Jos\'e received postdoctoral grants from the Brazilian funding agencies CNPq (Proc. 153903/2018-0), FAPERJ (Proc. E-26/202.011/2020) and FAPESP (Proc. 2020/10796-0), Gabriel is supported by a CAPES PhD scholarship and Simon was partially supported by FAPERJ (Proc.~201.194/2022) and by CNPq (Proc.~307521/2019-2).\\ This study was financed in part by the Coordenação de Aperfeiçoamento de Pessoal de Nível Superior – Brasil (CAPES) – Finance Code 001}

	\begin{abstract}
		Let $N_{\triangle}(G)$ be the number of triangles in a graph $G$.  In~\cite{GGS} and~\cite{NRS22} (respectively) the following bounds were proved on the lower tail behaviour of triangle counts in the dense Erd\H{o}s-R\'enyi random graphs $G_{(m)}\sim G(n,m)$:
		\[ \pr{N_{\triangle}(G_{(m)}) \, <  \, (1-\delta)\Ex{N_{\triangle}(G_{(m)})}} \,=\, \begin{cases} 
			\expb{-\Theta\left(\delta^2n^3\right)} & \text{if $n^{-3/2}\ll \delta\ll n^{-1}$} \\  
			\expb{-\Theta(\delta^{2/3}n^2) } & \text{if $n^{-3/4} \ll \delta \ll 1$.}
		\end{cases} 
		\phantom{\Bigg|}
		\]
Neeman, Radin and Sadun~\cite{NRS22} also conjectured that the probability should be of the form $\expb{-\Theta\left(\delta^2n^3\right)}$ in the ``missing interval'' $n^{-1}\ll \delta\ll n^{-3/4}$.  We prove this conjecture.

As part of our proof we also prove that some random graph statistics, related to degrees and codegrees, are normally distributed with high probability.
	\end{abstract}

	\maketitle

\section{Introduction}\label{sec : Intro}
	
The distribution and tail behaviour of subgraph counts, and especially the triangle count, has been a very active area of research in recent decades.  Subgraph counts are very natural examples of sums of dependent random variables.

In particular, a great many results have been proved regarding small deviations (of the order of the standard deviation) beginning with Ruci\'nski \cite{Ru88}, see also~\cite{BKR,JanRSA,Jan,JN,KRT,RR,Rol}.  There have also been many results which focus on large deviations (of the order of the mean) including the seminal articles of Vu~\cite{Vu} and Janson and Ruci\'nski~\cite{JR} in the early 2000s.  Later, Chatterjee and Varadhan~\cite{CV} related such deviations in dense random graphs to solutions of variational problems.  See~\cite{Au,CD20,El} for further developments related to these techniques.  This variational problem was solved for cliques by Lubetzky and Zhao~\cite{LZ} and in general by Bhattacharya, Ganguly, Lubetzky and Zhao~\cite{BGLZ}.   See the survey of Chatterjee~\cite{Chat} and the references therein for a detailed overview.    A major breakthrough by Harel, Mousset and Samotij~\cite{HMS2019} essentially resolved the large deviation upper tail problem for triangles.

Large deviations have also been studied with respect to the lower tail.  Zhao~\cite{Z} solved the lower tail variational problem for a large range in dense random graphs, where $p\in (0,1)$ is constant.  Kozma and Samotij~\cite{KS} defined an optimisation problem, whose solution would essentially determine the log probability for lower tail large deviations.  This characterisation works down to densities which are $\omega(n^{-1/m_2(H)})$ where $m_2(H)$ is the so called $2$-density of $H$.  As $m_2(K_3)=2$, this framework only works for $p\gg n^{-1/2}$ for triangles.  Recently, Jenssen, Perkins, Potukuchi and Simkina~\cite{JPPS} found the asymptotic rate for certain ranges of large deviations in the lower tail for triangle, in the ``critical'' range $p=\Theta(n^{-1/2})$.  See also, Janson and Warnke~\cite{JW} for some general lower tail results.
	
There has also been some interest in deviations of intermediate value, which we call moderate deviations.  These deviations are considered in the $G(n,p)$ model in~\cite{DE,DE2,FGN}.  It is argued by the third author, together with Goldschmidt and Scott~\cite{GGS} that, for many moderate deviation problems, the $G(n,m)$ model is more appropriate as it is possible to study finer causes of deviations, and that, in any case, one may deduce results for $G(n,p)$ by a simple conditioning argument.  See also~\cite{AGG}, which extends these results to sparser random graphs.
	
Let us now consider the model $G_{(m)}\sim G(n,m)$, in which $G_{(m)}$ is selected uniformly from graphs with $n$ vertices and $m$ edges.  Let $N_{\triangle}(G)$ be the number of triangles in a graph $G$.  The majority of results previously mentioned have focused on the upper tail, whereas we shall focus on the lower tail.  That is, we consider the question of how likely it is that a random graph has many \emph{fewer} triangles than expected.  Let us begin by stating the main results of ~\cite{GGS} and~\cite{NRS22} in this context.

Let $p\in (0,1)$ be a constant, and suppose that the sequence $m=m_n$ is such that\footnote{We shall omit floor signs and assume certain quantities are integers.  It will be clear that the results and proofs are robust to a $\pm 1$ additive error, which corresponds to $\pm O(n^{-2})$ error in terms of the density $p$.} $m=p\binom{n}{2}$.  That is, the random graph $G_{(m)}\sim G(n,m)$ has density $p$.  Let us also abbreviate $\Ex{N_{\triangle}(G_{(m)})}$ to $\mu_{n,m}$.   Note $\mu_{n,m}=p^3\binom{n}{3}$.  The main result of~\cite{GGS} in this context states that\footnote{In verifying that this corresponds to the result of~\cite{GGS}, note that triangles are counted with multiplicity in that article.}
\[
\pr{N_{\triangle}(G_{(m)}) \, <  \, (1-\delta)\mu_{n,m}}\, =\, \expb{\frac{-(1+o(1))\delta^2 \mu_{n,m}}{2(1-p)^2(1+2p)}}\, 
\]
provided $n^{-3/2}\ll \delta\ll n^{-1}$.  This shows that the ``normal tail'' extends far beyond deviations of the order of the standard deviation, $\Theta(n^{3/2})$ (for which $\delta=\Theta(n^{-3/2})$).

In ~\cite{NRS22} the authors used techniques associated with the spectrum of the random graph to prove the following results.  Writing $\gamma_p$ for $\log(p/(1-p))/2(2p-1)$ for $p>1/2$ and $\gamma_{1/2}=1$, then
\[
\pr{N_{\triangle}(G_{(m)}) \, <  \, (1-\delta)\mu_{n,m}}\, =\, \expb{-(1+o(1))\gamma_p \delta^{2/3}p^2n^2}\, 
\]
for $p\in [1/2,1)$ and $n^{-3/4}\ll \delta\ll 1$.  On the other hand, for $p\in (0,1/2)$ they prove that the asymptotic rate lies in the (closed) interval from $\gamma_p \delta^{2/3}p^2n^2$ to $\delta^{2/3}p^2n^2/2p(1-p)$.  

More crudely, if we ignore the exact constants, we may summerise the results of the two articles as follows.  Fix $p\in (0,1)$, and let $G_{(m)}\sim G(n,m)$ for $m=p\binom{n}{2}$ then 
\[ \pr{N_{\triangle}(G_{(m)}) \, <  \, (1-\delta)\mu_{n,m} }\,=\, \begin{cases} 
			\expb{-\Theta\left(\delta^2n^3\right)} & \text{if $n^{-3/2}\ll \delta\ll n^{-1}$} \\  
			\expb{-\Theta(\delta^{2/3}n^2) } & \text{if $n^{-3/4} \ll \delta \ll 1$\, .}
		\end{cases} 
		\phantom{\Bigg|}
		\]
		
In the ``missing regime'' where $n^{-1}\ll \delta\ll n^{-3/4}$ the article~\cite{NRS22} provides an upper bound of the form
$\expb{-\Omega\left(\delta^2n^3\right)}$ and a lower bound of the form $\expb{-O(\delta^{2/3}n^2)}$, see the discussion in Section 3.1 of~\cite{NRS22}.  These bounds are shown as the green and red lines in Figure~\ref{fig:graphic} respectively.

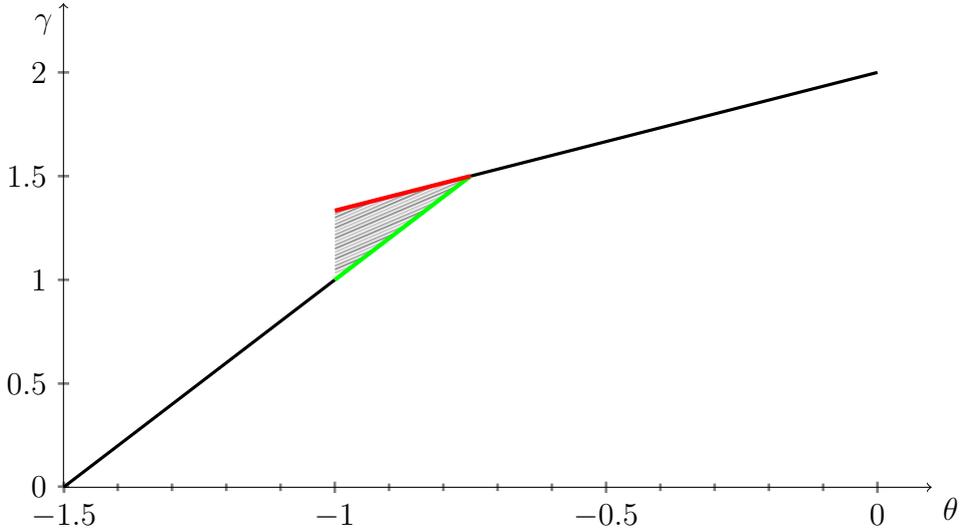
\begin{figure}\label{fig:graphic}
		\begin{tikzpicture}
			
			

			\begin{axis}[xmin=0,ymin=0,xmax=160,ymax=70,axis lines=middle, 
				standard,  axis line style={->},
				xlabel=$\theta$,
				ylabel=$\gamma$,
				minor xtick={0.1,10,20,30,40,60,70,80,90,110,120,130,140},
				tick style={line width=1pt},
				xtick={0.1,50,100,150},
				xticklabels={$\vspace{1mm} -1.5$, $-1$, $-0.5$, $0$},
				ytick={0.1,15,30,45,60},
				yticklabels={$0$,$0.5$,$1$,$1.5$,$2$}]

				\addplot[name path=GGS,very thick] coordinates {(0,0) (50,30)};
				\addplot[name path=NRS,very thick] coordinates {(75,45) (150,60)};

				\begin{scope}                                                 
					\clip (50,30) -- (75,45) -- (50,40) -- cycle;
					\foreach \y in {30, 33,...,60} {\addplot[thick,dmgr] coordinates{(50,\y) (80,10+\y)};}
					\foreach \y in {31, 34,...,60} {\addplot[thick,lgr] coordinates{(50,\y) (80,10+\y)};}
					\foreach \y in {32, 35,...,60} {\addplot[thick,mgr] coordinates{(50,\y) (80,10+\y)};}
					\foreach \y in {32.5, 35.5,...,60} {\addplot[thick,lgr] coordinates{(50,\y) (80,10+\y)};}
					\foreach \y in {33.5, 36.5,...,60} {\addplot[thick,mgr] coordinates{(50,\y) (80,10+\y)};}
					\foreach \y in {31.5, 34.5,...,60} {\addplot[thick,dmgr] coordinates{(50,\y) (80,10+\y)};}
					
				\end{scope}    
				
				\addplot[name path=missingU, color=green, ultra thick] coordinates {(50,30) (75,45)};
				\addplot[name path=missingL, color=red, ultra thick] coordinates {(50,40) (75,45)};

				
				
			\end{axis}
			
		\end{tikzpicture}		
		\caption{In this figure, the $\theta$-axis parameterises $\delta$ as $\delta=n^{\theta}$, and the $\gamma$-axis parameterises the log-probability $-\log{\pr{N_{\triangle}(G_{(m)})<(1-\delta)\mu_{n,m}}}=n^{\gamma}$.  The black lines show the results proved in~\cite{GGS} and~\cite{NRS22} respectively.  The green and red lines show the upper and lower bounds of~\cite{NRS22} in the interval $-1\le \theta\le -3/4$.  Our main theorem shows that the log-probability follows the green line in this interval.} 
	\end{figure}
	
Furthermore, they conjectured that the correct order in this range is $\delta^2n^3$.  Our main result confirms that this is indeed the case.

\newpage

\begin{theorem}\label{thm:main}	 
		Let $\lambda \in (0,1)$.  There exists a constant $C>0$ such that the following holds.  Let $\delta=\delta_n$ be such that $n^{-1}\le \delta\le C^{-1}n^{-3/4}$ and let $m=m_n$ be such that $\lambda \binom{n}{2}\le m\le (1-\lambda)\binom{n}{2}$, then
		\[ 
		\pr{N_{\triangle}(G_{(m)}) \, < \, (1-\delta) \mu_{n,m}} \, \ge\, \exp(-C\delta^2 n^3)\, .
		\]
	\end{theorem}
	
There is a sense in which our result is natural, as it shows that the log probability continues along its existing path up until $\delta=\Theta(n^{-3/4})$, as shown in Figure~\ref{fig:graphic}.  However, this was not at all a forgone conclusion.  In a surprising result, ~\cite{NRS22} showed that for $C_5$, and longer odd cycles, their results hold down to $\delta = \Theta(n^{-1})$, at which point there is a jump discontinuity in the log-probability.

The behaviour for the even cycle $C_4$ is more extreme, as a well known calculation using the Cauchy-Schwarz inequality shows that a deviation below $(1-\delta)\mu$ is \emph{impossible} for some $\delta=\Theta(n^{-1})$.  The situation for other bipartite graphs may well be similar, and is related to Siderenko's conjecture~\cite{Sid}.  

On the other hand, large deviations are always possible for counts of non-bipartite graphs $H$, since there is always some probability that $G_{(m)}$ contains a linear sized complete bipartite subgraph, which causes there to be fewer than $(1-\delta)\mu$ copies of $H$ for some constant $\delta>0$. Again, we refer the reader to Kozma and Samotij~\cite{KS} for a more details.

The problem of lower tail moderate deviations, for general graphs $H$, in the range $n^{-1}\ll \delta\ll 1$, is wide open.  And it would be of interest to determine the behaviour of lower tails of subgraph counts in general across this range.  Related to the ``continuity'' question, we make the following conjecture -- there is a ``discontinuity'' in the lower tail log probability for the graph $H$, around $\delta=n^{-1}$, if and only if $H$ is triangle-free.  Let $N_{H}(G)$ denote the number of copies of $H$ contained in $G$.  Here is a more formal statement.

\begin{conjecture}\label{conj:cont} Let $p\in (0,1)$ be a constant, and let $m=p\binom{n}{2}$.   Let $\mu_{H}=\Ex{N_H(G_{(m)})}$.  Then the function
\[
\beta\,\, \mapsto\, \, \frac{\log{\log{\pr{N_{H}(G_{(m)})\, <\, (1-n^{\beta})\mu_H}^{-1}}}}{\log{n}}
\]
has a jump discontinuity at $\beta=-1$, if and only if $H$ is triangle-free.
\end{conjecture}

It may well be of interest to investigate similar lower tail problems for other discrete combinatorial problems, such as arithmetic progressions in uniform random sets.  Some results may be found in~\cite{Irreg} and~\cite{Reg} for this model.  We remark that lower tail probabilities are often much larger in the Binomial random model, see~\cite{JW}.

In the following section we give an overview of the proof, and a summary of the layout of the article.

\section{An overview of the proof}\label{sec:overview}

Our approach starts with the naively simple idea that fewer triangles are created if many of the edges added are in few triangles (i.e., have small codegree).  However, it turns out to be surprisingly difficult to make this approach rigorous, and so we shall consider (and condition on) a slightly different event.

We now attempt to explain intuitively how and why our proof works.  We also introduce notation that will be used throughout the article.  Set $N:=\binom{n}{2}$.  We shall use the notation $\pm K$ to denote an error of up to $|K|$, so that $a=b\pm K$ means $a\in [b-K,b+K]$.

Our objective is to prove a lower bound on the probability, 
\[
\pr{N_{\triangle}(G_{(m)}) \, < \, (1-\delta)\mu_{n,m}}\, ,
\]
that $G_{(m)}$ contains many fewer triangles than expected, where $n^{-1}\le \delta\le C^{-1}n^{-3/4}$.  We achieve this by defining an event, which is not too rare, and on which the conditional expectation of the triangle count is much smaller.  Given an event $E$ such that 
\[
\Ex{N_{\triangle}(G_{(m)})\, |\, E}\, \le \, (1-2\delta)\mu_{n,m}
\]
It follows from Markov's inequality that
\[
\pr{N_{\triangle}(G_{(m)}) \, \ge \, (1-\delta)\mu_{n,m} | E}\, \le \, \frac{(1-2\delta)\mu_{n,m}}{(1-\delta)\mu_{n,m}}\, \le\, 1-n^{-1}\, .
\]
Considering the complementary event, we have $\pr{N_{\triangle}(G_{(m)}) \, < \, (1-\delta)\mu_{n,m} | E}\ge n^{-1}$, from which it follows that 
\[
\pr{N_{\triangle}(G_{(m)}) \, < \,(1-\delta)\mu_{n,m}}\ge n^{-1}\pr{E}\, .
\]
In summary, to prove Theorem~\ref{thm:main} it suffices to prove the following proposition.

\begin{restatable}{prop}{propmain}\label{prop:main} Let $\lambda \in (0,1)$.  There exists a constant $C>0$ such that the following holds.  Let $\delta=\delta_n$ be such that $n^{-1}\le \delta\le C^{-1}n^{-3/4}$ and let $m=m_n$ be such that $\lambda N\le m\le (1-\lambda)N$, then there exists an event $E$ such that
	\begin{enumerate}
		\item[(i)] $\pr{E}\, \ge \, \exp(-C\delta^2 n^3)$ and
		\item[(ii)] $\Ex{N_{\triangle}(G_{(m)})|E}\, \le\, (1-2\delta)\mu_{n,m}$.
	\end{enumerate}
\end{restatable}

What kind of event $E$ might ``cause'' there to be fewer triangles?

We may clearly define $G_{(m)}$ by running the Erd\H{os}-R\'enyi random graph process, which starts with an empty graph on $n$ vertices and to which we add an edge at each step.  It is easy to see that $G_{(m)}$ is the $m$th graph in this process.

Here is a first attempt at defining such an event $E$.  Run the process for some time $m_0$, a little smaller than $m$, and let $G^0$ be the current graph (with $m_0:=p_0N$ edges).  Let $H$ be the graph which represents the non-edges of $G^0$ with relatively small codegree.  Now, a good way to create fewer triangles, in the remainder of the process, would be to include in $G\setminus G^0$ more pairs from $H$ than expected.  Indeed, pairs with smaller codegree create fewer triangles!  This suggests taking $E$ to be the event that the $e(G\cap H)$ is significantly larger than its mean.

However, when calculating the conditional expectation $\Ex{N_{\triangle}(G_{(m)})|E}$ we must consider not only the triangles with two edges in $G^0$ and one ``new'' edge, but also the triangles with two or three ``new'' edges.  The conditional expectations of these latter two turn out to be challenging to calculate.

In fact, the calculations become easier if the graph (in the role of the graph $H$ above) is close to regular.  For this reason we turn to a new parameter, related to codegree, which we call \emph{synergy}. 
	
\begin{definition}
For $p\in(0,1)$, the \emph{p-synergy} of the vertices $u$ and $w$ in a $n$ vertex graph $G$, is 
\begin{equation}\label{synergy}
S^p_G(u,w)\, :=\, d_G(u,w)-pd_G(u)-pd_G(w)+p^2(n-2)\, .
\end{equation}

If $p$ is omitted from the notation, as in $S_{G}(u,w)$, then take $p=e(G)/N$.

If $G$ is clear from context then it too may be omitted, as in $S(u,w)$.
\end{definition} 

Our approach will essentially be as described above, except using synergies instead of codegrees.  The reader will note that synergies are essentially recentered versions of codegrees.  This recentering is useful for some technical reasons.  A key facet of the article, see Section~\ref{sec:syn}, is to show that the sequence of synergies $S(u,w)$ over the non-neighbours, $w\in V(G)\setminus N_{G^0}(u)$, of a given vertex $u$, is, with very high probability, close to normally distributed.

What then is our new attempt at defining the event $E$?  

As before we first run the process for $m_0$ steps to reveal a graph $G^0$ (with $m_0$ edges).  We then let $f_1,\dots f_{N-m_0}$ be all the non-edges of $G^0$ in non-decreasing order of synergy $S(u,w)$.  We may define
\[
F_-\, :=\,\left\{f_i:i\in\left\{1,\dots,\frac{N-m_0}{2}\right\}\right\}\, ,
\]
that is $F_-$ is a graph consisting of the half\footnote{We shall always take $m_0$ such that $N-m_0$ is even.} of non-edges of $G^0$ with smaller synergy. Again it is true that taking pairs of $F_-$ as edges will tend to create fewer triangles.  For this reason we may want to take $E$ to be the event that $e(G^1\cap F_-)$ is much larger than its expected value, where $G^1=G_{(m)}\setminus G^0$ is the graph of edges added in the remainder of the process. 

This is essentially the event $E$ we shall take.  For technical reasons we shall
\begin{enumerate}
\item[(i)] find it useful to include in $E$ some basic quasirandomness properties of the first graph $G^0$, and
\item[(ii)] use a parameter $\alpha$, so that the event is actually $E_{\alpha}$, for some $\alpha>0$.
\end{enumerate}

We now describe the layout of the article.  In Section~\ref{sec:syn}, we prove properties about synergies of pairs of vertices.  In Section~\ref{sec:CoSy}, we prove important results about the graph $F_{-}$, and about the connection between synergies and codegrees.  In Section~\ref{sec:event}, we first define an event $E_0$ which encodes the basic quasiradnomness properties we require of $G^0$ and then we define the events $E_{\alpha}$.  We also prove a lower bound on the probability of $E_{\alpha}$, which corresponds to part (i) of Proposition~\ref{prop:main}.  

Finally, in Section~\ref{sec:proof}, we show that conditioning on the event $E_{\alpha}$ has the correct effect on the expected number of triangles.  Doing so we complete the proof of Proposition~\ref{prop:main}, and therefore complete the proof of Theorem~\ref{thm:main}.  We remark that the main challenge in Section~\ref{sec:proof} is to control the influence of triangles with two or three ``new'' edges, i.e., two or three edges in $G^1$.  In some sense our control of their influence is just about good enough, and so to reduce their influence by a constant factor we take $m_0=(1-\eta)m$ and $m_1:=m-m_0=\eta m$, where we shall choose the constant $\eta>0$ to be sufficiently small.  

\subsection*{Notation}
Throughout $N$ denotes $\binom{n}{2}$ and $\Phi(t)$ is the distribution function of a standard normal random variable.  

Given a graph $G$ and a vertex $u\in V(G)$ we write $N_G(u)$ (or simply $N(u)$) for the neighbourhood of $u$.  We write $d_G(u)$ (or simply $d(u)$) for the degree of $u$.  For a pair $u,w\in V(G)$, we denote by $d_G(u,w)$ (or simply $d(u,w)$) the codegree of $u$ and $w$.  The following table contains quantities which we define later in the paper, and may be useful for reference.

\begin{table}[h]
\begin{tabularx}{\textwidth}{ll}
\toprule
  Synergies\phantom{\bigg{|}} & \hspace{-22mm} First Introduced\\
  $S^p_G(u,w)\, :=\, d_G(u,w)-pd_G(u)-pd_G(w)+p^2(n-2)\phantom{\Big{|}}$ & \ref{synergy} \\
  $\sigma(p)^2\, :=\, \Var\big(S_{G_p}^p(u,w)\mid uw\not\in E(G_p)\big)\phantom{\Big{|}}$\ & \ref{eq:sigmap}\\
  $\tilde{S}^p(u,w)\, :=\, S^p(u,w)/\sigma(p)\phantom{\Big{|}}$& \ref{normalized-synergy}\\
  $\S^{p}_{G}(u)\, :=\, \big(\tilde{S}^{p}_G(u,w)\, :\, w\in V(G)\setminus (N(u)\cup \{u\})\big)\phantom{\Big{|}}$& \ref{eq:psynergy-vector}\\
  $\S_{G}(u)\, :=\, \big(\tilde{S}_G(u,w)\, :\, w\in V(G)\setminus (N(u)\cup \{u\})\big)\phantom{\Big{|}}$&\ref{eq:synergy-vector}\\
  Synergy relative to a subset $I\subseteq V(G)$\phantom{\bigg{|}}\\
  $S_{G}^p(u,w|I)\, :=\, d_{G}(u,w)\, -\, p d_{G}(u)\, -\, p|N_{G}(w)\cap (V\setminus I)|\,  +\, p^2(n-|I|-1)\phantom{\Big{|}}$& \ref{eq:relative-synergy}\\
  Distances between distributions\phantom{\bigg{|}}\\
  $F_{\x}(t)\, :=\, |\{i:x_i\leq t\}|/k \qquad\qquad (\text{for } \x\in \RR^k)\phantom{\Big{|}}$& \ref{eq:xquant}\\
  $d_t(\x,\Phi)\, :=\, \left|F_{\x}(t)-\Phi(t)\right|\phantom{\Big{|}}$ & \ref{tdistance} \\
  $d(\x,\Phi)\, :=\, \sup\{d_t(\x,\Phi)\, :\, t\in \RR\}\phantom{\Big{|}}$& \ref{distance} \\
  $d_t(X,\Phi)\, :=\, d(F_X(t),\Phi(t))\phantom{\Big{|}}$& \ref{eq:tdistance-rv}\\
  $d(X,\Phi)\, :=\, \sup\{d_t(X,\Phi)\, :\, t\in \RR\}\phantom{\Big{|}}$& \ref{eq:distance-rv}\\
\end{tabularx}
\end{table}
\begin{table}[h]
\begin{tabularx}{\textwidth}{ll}
  Degree and Codegree deviations\phantom{\bigg{|}}\\
  $D(u)\, :=\, d(u)-p(n-1)\phantom{\Big{|---------------------}}$& \ref{eq:degree-deviation}\\
  $D(u,w)\, :=\, d(u,w)-p^2(n-2)\phantom{\Big{|}}$& \ref{eq:codegree-deviation}\\
\bottomrule
\end{tabularx}
\end{table}

	
	

	\section{The Synergies are normally normal} \label{sec:syn}

 In this section we prove that, for each vertex $u$, it is very likely that the sequence of synergies $(S_{G^0}(u,w): w\in V(G)\setminus (N_{G^0}(u)\cup \{u\}))$ is close to normally distributed.

Given a vector $\x=(x_1,\dots,x_k)$ we define
	\eq{xquant}
	F_{\x}(t)\, :=\, \frac{|\{i:x_i\leq t\}|}{k}\, ,
	\eqe
	which may be thought of as a distribution function associated with the vector $\x$.  In this context it is natural to consider the distance between this distribution and the distribution function $\Phi$ of the standard normal.  For each $t\in \RR$ we may set
	\begin{equation}\label{tdistance}
	d_t(\x,\Phi)\, :=\, \left|F_{\x}(t)-\Phi(t)\right|
	\end{equation}
	and 
	\begin{equation}\label{distance}
	d(\x,\Phi)\, :=\, \sup\{d_t(\x,\Phi)\, :\, t\in \RR\}\, .
	\end{equation}
In light of these definitions we say that $\x$ is $\eps$-\emph{close} to normal if $d(\x,\Phi)\le \eps$, and $\eps$-\emph{far} from normal if $d(\x,\Phi)>\eps$.

We shall only consider synergies of pairs $u,w$ which are non-neighbours.  Note that $\Ex{S_{G_p}^p(u,w)\mid uw\not\in E(G_p)}=0$, where $G_p\sim G(n,p)$.  Let us also observe the value of the variance
	\eq{sigmap}
	\sigma(p)^2\, :=\, \Var\big(S_{G_p}^p(u,w)\mid uw\not\in E(G_p)\big)\, =\, p^2(1-p)^2(n-2)\, ,
	\eqe
 see the appendix for a proof of this statement.
	  We may now introduce the normalised version of synergy:
	\begin{equation}\label{normalized-synergy}
	    \tilde{S}^p(u,w)\, :=\,  \frac{S^p(u,w)}{\sigma(p)}\, .
	\end{equation}
	Again, we may also write $\tilde{S}_G^p(u,w)$ to clarify which graph is considered.  If $p$ is not specified, as in $\tilde{S}(u,w)$ or $\tilde{S}_G(u,w)$, then the normalisation\footnote{In fact, if $G\sim G(n,m)$ then this normalisation is in some sense ``incorrect'' as we are using the mean and standard deviation associated with the $G(n,p)$ model.  However, as these parameters differ by so little, between the two models, the same normalisation may be used, and this simplifies comparisons between the models.} uses $\sigma(p)$ where $p=e(G)/N$.
	
	Now we can state the main theorem of this section.  Given a vertex $u$ of a graph $G$, let 
	\eq{psynergy-vector}
	\S^{p}_{G}(u)\, :=\, \big(\tilde{S}^{p}_G(u,w)\, :\, w\in V(G)\setminus (N(u)\cup \{u\})\big)\, .
	\eqe
	and
	\eq{synergy-vector}
	\S_{G}(u)\, :=\, \big(\tilde{S}_G(u,w)\, :\, w\in V(G)\setminus (N(u)\cup \{u\})\big)\, .\phantom{\Big|}
	\eqe
	That is, $\S^{p}_{G}(u)$ and $\S_{G}(u)$ are sequences of (normalised) synergies of pairs $uw$ where $w$ varies over non-neighbours of $u$ in $G$.

	
	\begin{prop}\label{prop:syngm}
		Let $\lambda>0$, let $0<\beta<1/4$ and let $\eps=n^{-\beta}$.  Let $G_{(m)}\sim G(n,m)$, where $\lambda N\le m\le (1-\lambda)N$.  Then, for each vertex $u\in V(G_{(m)})$,
		\[
		\pr{\S_{G_{(m)}}(u)\, \text{is } \eps\text{-far from normal}}\, \le\, \exp(-n^{1/2-2\beta-o(1)})\, =\,  \exp(-n^{\Omega(1)})\, .
		\]
		\end{prop}
		
		We shall begin by showing that this result follows from the analogous statement for $G\sim G(n,p)$.
		
		\begin{prop}\label{prop:syngp}
		Let $\lambda>0$ and let $p\in [\lambda,1-\lambda]\subseteq (0,1)$.  Let $0<\beta<1/4$ and let $\eps=n^{-\beta}$.  Let $G_p\sim G(n,p)$.  Then, for each vertex $u\in V(G_{p})$,
		\[
		\pr{\S^{p}_{G_p}(u)\, \text{is } \eps\text{-far from normal}}\, \le\, \exp(-n^{1/2-2\beta-o(1)})\, =\,  \exp(-n^{\Omega(1)})\, .
		\]
		\end{prop}

  \begin{remark} It may well be possible to prove that $\S^{p}_{G_p}(u)$ is $\eps$-close to normal for smaller values of $\eps$, perhaps as small as $O(n^{-1/2})$,  by using properties of the relative synergy $S^p(u,w|I)$ (defined below) and the difference $S^p(u,w)-S^p(u,w|I)$ discussed in Claim 2 of the proof of Proposition~\ref{prop:syngp}.
  \end{remark}

The advantage of the $G(n,p)$ statement is that $G_p\sim G(n,p)$ has more independence.  Let us first observe that Proposition~\ref{prop:syngm} follows from Proposition~\ref{prop:syngp}.
		
\begin{proof}[Proof of Proposition~\ref{prop:syngm}]
Given $m\in [\lambda N, (1-\lambda)N]$, set $p:=m/N\in [\lambda,1-\lambda]$.  We write $\mathbb{P}_p$ and $\mathbb{P}_m$ for the probability measures associated with $G\sim G(n,p)$ and $G\sim G(n,m)$ respectively.  As there is probability at least $n^{-2}$ that $G_p$ has exactly $m$ edges, and such a graph is distributed as $G(n,m)$, we have, for any event $E$ of graphs with $m$ edges:
\[
\mathbb{P}_p(E)\, \ge\, n^{-2}\mathbb{P}_m(E)\, .
\]
Using this inequality and Proposition~\ref{prop:syngp}, we have that
\begin{align*}
\pr{\S_{G_{(m)}}(u)\, \text{is } \eps\text{-far from normal}}\, &=\, \mathbb{P}_m(\S_{G}(u)\, \text{is } \eps\text{-far from normal})\phantom{\big|}\\
&=\, \mathbb{P}_m(\S^p_{G}(u)\, \text{is } \eps\text{-far from normal})\phantom{\big|}\\
&\le\, n^2\mathbb{P}_p(\S^p_{G}(u)\, \text{is } \eps\text{-far from normal})\phantom{\big|}\\
&\le\, n^2\exp(-n^{1/2-2\beta-o(1)})\, .\phantom{\big|}
\end{align*}
As $n^2$ is of the form $\exp(n^{o(1)})$, this gives the required bound.
\end{proof}

We now turn to the task of proving Proposition~\ref{prop:syngp}.  We must prove that, with very high probability, the vector $\S^{p}_{G}(u)$ is close to normal.  For the remainder of the section $G\sim G(n,p)$.

The main challenge is that, even in $G(n,p)$, the (normalised) synergies $\tilde{S}^{p}_{G}(u,w)$, which are the entries of $\S^{p}_{G}(u)$, are not independent.  Our task would be significantly easier if they were independent, see Lemma~\ref{lem:indep}, or even if they could be closely approximated by a sequence of independent random variables.

While we are not able to approximate the whole sequence $\S^{p}_{G}(u)$ by a sequence of independent random variables, we are able to do so if we first restrict to a subsequence -- a small random sample of the entries of $\S^{p}_{G}(u)$.

In Section~\ref{subsec : indep}, we prove a lemma (Lemma~\ref{lem:indep}) about the case of a sequence of independent random variables.  In Section~\ref{subsec : sampling}, we prove a lemma (Lemma~\ref{lem:samp1}) which implies that a sampling can be used to control the whole.  Finally, we show in Section~\ref{subsec : synergies} that these results imply that $\S^{p}_{G}(u)$ is likely to be close to normal, proving Proposition~\ref{prop:syngp}.

We remark that synergies are not really so complex as random variables, especially after we reveal one of the neighbourhoods.  This is illustrated in Figure~\ref{fig:syndraw}.


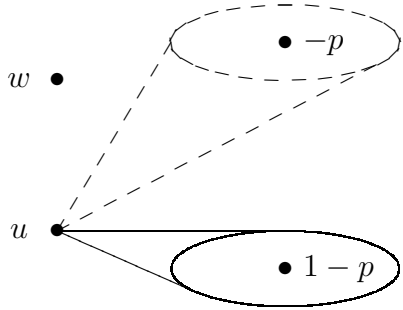
\begin{figure}\label{fig:syndraw}
		$$	
		\begin{xy}
			(0,0)*{\bullet};
			(0,20)*{\bullet};
			(-5,0)*{u};
			(-5,20)*{w};
			(30,-5)*\xycircle(15,5){};
			(0,0)*{}; (18,-8)*{} **\dir{-};
			(0,0)*{}; (30,0)*{} **\dir{-};
			(30,25)*\xycircle(15,5){--};
			(30,-5)*{\bullet};
			(35,25)*{-p};
			(37,-5)*{1-p};
			(30,25)*{\bullet};
			(0,0)*{}; (15,25)*{} **\dir{--};
			(0,0)*{}; (42,22)*{} **\dir{--};
		\end{xy}
		$$
		\caption{In this figure we illustrate the fact that, conditioned on the neighbourhood of $u$, $N_{G_p}(u)$, each edge from $w$ adds $1-p$ to the synergy $S_{G_p}^p(u,w)$ if the other end is in $N_{G_p}(u)$ and $-p$ if not.}
	\end{figure}

\subsection{Sequences of iid random variables}\label{subsec : indep}

The definition of distance from normal applies even more naturally to random variables.  Let $F$ be the distribution function of a random variable $X$.  Then we may define
\eq{tdistance-rv}
d_t(X,\Phi)\, :=\, \left|F(t)-\Phi(t)\right|
\eqe
and 
\eq{distance-rv}
d(X,\Phi)\, :=\, \sup\{d_t(X,\Phi)\, :\, t\in \RR\}\, .\phantom{\big|}
\eqe
In fact, this is known as the Kolmogorov distance between the distributions.

\begin{lem}\label{lem:indep}
Let $\eps>0$ and $k\ge 1$.  Let $X$ be a random variable with $d(X,\Phi)\le \eps$.  Let $\XX=(X_1,\dots ,X_k)$ be a sequence of iid copies of $X$.   Then
\[
\pr{\XX \text{ is } 3\eps\text{-far from normal}}\, \le\,  2\eps^{-1}\, \exp\left(\frac{-\eps^2 k}{2}\right)\, .
\]
\end{lem}

\begin{proof} 
We first prove for a fixed $t\in \RR$ that 
\eq{fixedt}
\pr{d_{t}(\XX,\Phi)\, >\, 2\eps}\, \le\, 2\exp\left(\frac{-\eps^2 k}{2}\right)\, .
\eqe
We then deduce the lemma by a simple $\eps$-net argument.

Fix $t\in \RR$.  Let $p(t):= \pr{X\le t}$.  As $d(X,\Phi)\le \eps$,  we have $|p(t)-\Phi(t)|\le \eps$.

Let $Y:=|\{i:X_i\le t\}|$.  Note that $Y\sim \Bin(k,p(t))$ and that $\Ex{Y}=kp(t)$, so that $|\Ex{Y}-k\Phi(t)|\le \eps k$.  We now have that
\[
\pr{d_{t}(\XX,\Phi)\, >\, 2\eps}\, =\, \pr{|Y-k\Phi(t) |>2\eps k}\, \le\, \pr{|Y-\Ex{Y}|> \eps k}\, .
\]
By applying the Chernoff bound for the lower tail to both $Y$ and $k-Y\sim \Bin(k,1-p(t))$, we obtain that
\[
\pr{|Y-\Ex{Y}|> \eps k}\, \le\, 2\exp\left(\frac{-\eps^2 k}{2}\right)\, ,
\]
as claimed.

We now proceed to the $\eps$-net argument.  For each $j=1,\dots , \lfloor \eps^{-1}\rfloor$, let $t_j=\Phi^{-1}(j\eps)$.  We also extend the definition to $t_0=-\infty$ and $t_{\lfloor \eps^{-1}\rfloor+1}=\infty$.    In a slight abuse of notation we extend our definition of $d_{t}(\XX,\Phi)$ to include $d_{-\infty}(\XX,\Phi)=d_{\infty}(\XX,\Phi)=0$ for any real vector $\XX$.
	
Using the notation $\mathbb{P}_{i\in I}$ for the uniform probability measure on $I$, it now follows by an easy monotonicity argument that for $t\in (t_{j-1},t_j]$ we have
\begin{align*}
	\mathbb{P}_{i\in I}\big(X_i\leq t\big)\, -\, \Phi(t)\, &\le\,  \mathbb{P}_{i\in I}\big(X_i\leq t_j\big)\, -\, \Phi(t_{j-1})\\
		& \le\, d_{t_j}(\XX,\Phi)\, +\, \eps\, .
\end{align*}
By an essentially identical argument
\[
\Phi(t)\, -\, \Pr_{i\in I}\big(X_i\leq t\big)\, \le\, d_{t_{j-1}}(\XX,\Phi)\, +\, \eps\, .
\]
It follows that
\eq{epsnet}
d(\XX,\Phi)\, \le\, \max_{j} d_{t_j}(\XX,\Phi)\, +\, \eps\, .
\eqe
In other words, if we accept losing an additive term $\eps$, we only need to consider the distance $d_t(\XX,\Phi)$ at the finitely many points $t=t_j$ for $j=1,\dots , \lfloor \eps^{-1}\rfloor$.  

We may now complete the proof, as, if $\XX$ is $3\eps$-far from normal, then by~\eqr{epsnet}, we have that $d_{t}(\XX,\Phi)\, >\, 2\eps$ for some $t\in \{t_j: j=1,\dots, \lfloor \eps^{-1}\rfloor\}$.  By~\eqr{fixedt} and a union bound, this has probability at most
\[
\eps^{-1} \, \cdot \, 2\exp\left(\frac{-\eps^2 k}{2}\right)\, ,
\]
as required.
\end{proof}


The case where each $X_i$ is itself obtained as a weighted sum of iid Bernoulli random variables will be particularly relevant.  Let $(a_i)_{i=1}^\ell$ be a sequence of real numbers and let $X:=\sum_{i=1}^{\ell}a_i\xi_i$, where the random variables $\xi_i$ are iid with distribution Bernoulli $\textrm{Bern}(p)$.  Let $\tilde{X}:=(X-\mu_X)/\sigma_X$, be the renormalisation of $X$, using $\mu_X=p\sum_i a_i$ and $\sigma_X^2\, =\, p(1-p)\sum_{i}a_i^2$.  By the Berry-Esseen inequality~\cite{Berry,Esseen} (using, for example, Theorem 1 of~\cite{Berry}, with $\Lambda=\max\{|a_i|\}$ and $\sigma=\sigma_X$) we have that
\eq{Berry}
d(\tilde{X},\Phi)\, \le\, \frac{2\max\{|a_i|\}}{\sigma_X}
\eqe
In particular, as a corollary of the above lemma we have.

\begin{cor}\label{empcor}
Let $(a_i)_{i=1}^\ell$ be a sequence of real numbers, let $p\in (0,1)$ and set $\mu=p\sum_{i}a_i$ and $\sigma^2=p(1-p)\sum_{i}a_i^2$.   For $j=1,\dots, k$, let us define $X_j=\sum_{i} a_i\xi_{i,j}$ where the $\xi_{i,j}$ are iid $\textrm{Bern}(p)$ random variables.  And let $\tilde{X}_j=(X_j-\mu)/\sigma$ for each $j=1,\dots ,k$.  Let $\XX=(\tilde{X}_1,\dots ,\tilde{X}_k)$.  Then for every $\eps\ge 2\max\{|a_i|\}\sigma^{-1}$, we have
	$$\mathbb{P}\big(\XX\text{ is }3\eps\text{-far from normal}\big)\,\leq\, 2\eps^{-1}\,\exp\left(\frac{-\eps^2k}{2}\right).$$
\end{cor}

\begin{proof} By~\eqr{Berry}, each random variable $\tilde{X}_j$ is $\eps$-close to normal, as $\eps\ge 2\max\{|a_i|\}\sigma^{-1}$.  They are also iid, and so the result follows from Lemma~\ref{lem:indep}.
\end{proof}

For example, if $|a_i|\le 1$ for all $i$, then putting $\eps=2\sigma^{-1}$, we would have
\[
\pr{\XX\text{ is }(6\sigma^{-1})\text{-far from normal}}\, \le\, \sigma\, \exp\left(\frac{-2k}{\sigma^2}\right)\, .
\]

As we have discussed, the entries in the vector of synergies $\S^{p}_{G}(u)$ are \emph{not} independent, and so we cannot directly deduce Proposition~\ref{prop:syngp} from Lemma~\ref{lem:indep}. Instead, we fix a subset $I$ of the non-neighbours of $u$, and for each $w\in I$ we consider the \emph{relative synergy} defined by
\eq{relative-synergy}
S_{G}^p(u,w|I)\, :=\, d_{G}(u,w)\, -\, p d_{G}(u)\, -\, p|N_{G}(w)\cap (V\setminus I)|\,  +\, p^2(n-|I|-1)\, .
\eqe
In some sense, this relative synergy contains ``less information'' than the original synergy $S_{G}^p(u,w)$, as we no longer ask how many neighbours $w$ has in $I$, instead replacing this value with its expected value, $p(|I|-1)$.  On the other hand, relative synergies are independent!

\begin{lem}
	Let $u\in G$, let $I\subseteq V(G)\setminus (N_G(u)\cup\{u\})$ and let $v,w\in I$. Then $S^p(u,v|I)$ and $S^p(u,w|I)$ are conditionally independent and identically distributed, given $N_G(u)$.
\end{lem}
\begin{proof}
    Each edge of $G$ is in it independently from the other edges, for any $x,y\in\mathbb{R}$. Therefore the events $\big[S^p(u,w|I)\leq x\big]$ and $\big[S^p(u,v|I)\leq y\big]$ are independent, since they depend on disjoint sets of edges (the edge $vw$ is not considered by the relative synergy).  It is clear from the definition that they are identically distributed.
\end{proof}
	
	\subsection{Sampling}\label{subsec : sampling}
	
	Given a vector $\x:=(x_i:i\in I)$ and a subset $J\subseteq I$, we write $\x_J:=(x_j:j\in J)$ for the vector restricted to this subset of the indices.  The following lemma states that if a vector is far from normal then the sample of it is also likely to be far from normal\footnote{There is nothing special about the normal distribution, but we state the result for normal as this is the case we require.  One could also prove stronger bounds of course.}.
	
	\begin{lem}\label{lem:samp1}
		Let $\eps>0$, let $1\le k\le K$ and suppose that $\eps^2 k\ge 2$.  Let $\x:=(x_1,\dots x_K)$ be a vector which is $2\eps$-far from normal, and let $I\subseteq [K]$ be a random subset chosen uniformly from $k$-element subsets of $[K]$.  Then
		\[
		\Pr\big(\x_{I}\text{ is } \varepsilon\text{-far from normal}\big)\, \ge\, \frac{1}{2}\, .
		\]
	\end{lem}

	\begin{proof}
	Reordering, if necessary, we may assume that $\x=(x_1,\dots, x_{K})$ is non-decreasing.  As $d(\x,\Phi)\ge 2\eps$, it follows that for some $t\in \RR$ we have $d_t(\x,\Phi)\ge 2\eps$.   Let $\ell$ be maximal such that $x_{\ell}\le t$.  The condition $d_t(\x,\Phi)\ge 2\eps$ gives us that 
	\eq{ellfar}
	\big|\ell -\Phi(t)K\big|\, >\, 2\eps K\, .
	\eqe
	
We now consider the vector $\x_{I}$.  Let $Y$ be the random variable $|I \cap[\ell]|$.  Note that $d_t(\x_{I},\Phi)=|Y-\Phi(t)k|/k$, and so the vector $\x_{I}$ is certainly $\eps$-far from normal in the case
\eq{tohyp}
|Y-\Phi(t)k|\, \ge\, \eps k\, .
\eqe 
We note that $Y$ is distributed as a hypergeometric random variable, $Y\sim\textrm{Hyp}(K,\ell,k)$, which has expected value $\mu=\ell k/K$, and so, by~\eqr{ellfar}, we have $|\mu-\Phi(t)k|\ge 2\eps k$.  Now, the probability $\x_{I}$ is \emph{not} $\eps$-far from normal is at most
\[
\pr{|Y-\Phi(t)k|\, <\, \eps k}\, \le\, \pr{|Y-\mu|\, >\, \eps k}\, .
\]
For hypergeometric random variables, the variance satisfies $\sigma^2\le \mu \le \ell k/K$.  And so, it follows from Chebychev's inequality that this probability is at most
\[
\frac{\ell k}{\eps^2 k^2 K}\, \le\, \frac{1}{\eps^2 k}\, \le\, \frac{1}{2}\, .
\]
\end{proof}

In particular, this may be applied to realisations of random vectors $\XX\in \RR^K$.  Doing so yields the following corollary.

\begin{cor}\label{cor:sample}
Let $\eps>0$, let $1\le k\le K$ and suppose that $\eps^2 k\ge 2$.  Let $\XX\in \RR^K$ be a random vector, and let $I\subseteq [K]$ be a random subset chosen uniformly from $k$-element subsets of $[K]$. Then
\[
\mathbb{P}(\XX\text{ is }2\eps\text{-far from normal})\,\leq\, 2 \,\mathbb{P}(\XX_I\text{ is }\eps\text{-far from normal})\, .
\]
\end{cor}

	\subsection{Proof of Proposition~\ref{prop:syngp}}\label{subsec : synergies}
	 Recall that $G\sim G(n,p)$.  We omit $G$ from degree and synergy notations for the remainder of the section.  We use the following notation 
  \eq{degree-deviation}
  D(u)\, :=\, d(u)-\mathbb{E}[d(u)]\, =\, d(u)-p(n-1)
  \eqe
  and 
  \eq{codegree-deviation}
  D(u,w)\,:=\, d(u,w)-\mathbb{E}[d(u,w)]\, =\, d(u,w)-p^2(n-2)\, .
  \eqe

Before giving the proof, let us explain why the sampling trick is really necessary. Given a vertex $u$, its neighbourhood $N(u)$ and a vertex $w\in V\setminus (N(u)\cup\{u\})$, we can write \begin{equation}\label{syneq1prop}
			S^p(u,w)\, +pd(u)\, -\, p^2(n-2)\, =\, d(u,w)\, -\, pd(w)\, =\, \sum_{v\in V\setminus\{u,w\}}a_v\xi_{wv}
		\end{equation}
		where each $\xi_v$ is an iid copy of $\Bern(p)$ and $a_v=1_{v\in N(u)}-p$ for all $v\in V(G)\setminus \{u,w\}$. We would like to apply Corollary \ref{empcor}.  The problem is that, if we set $X_w$ equals to the sum in~\eqref{syneq1prop}, then the sums do not use disjoint sets of iid Bernoulli random variables as we vary over $w$.  So the sampling idea will be vital.

	\begin{proof}[Proof of Proposition \ref{prop:syngp}]
		Fix a vertex $u\in G$ and reveal the neighbourhood $N(u)$.  Recall $\eps=n^{-\beta}$, for some $\beta\in (0,1/4)$.

Let $k=n^{1/2}$. Let $I$ be a uniformly chosen $k$-element subset of $(N(u)\cup\{u\})^c$.  By Corollary~\ref{cor:sample} it suffices to prove that, except with probability at most $\exp(-n^{1/2-2\beta+o(1)})$, we have that
\eq{sampleS}
\big(\tilde{S}^p(u,w)\, :\, w\in I\big)
\eqe
is $(\eps/2)$-close to normal.
  
Before tackling this problem directly, we recall and introduce some notation, and prove some auxiliary results about these quantities.  

For each $w\in I$, we recall the relative synergy
		$$S^p(u,w|I)\, :=\, d(u,w)\, -\, pd(u) \, -\, p\big|N(w)\cap (V\setminus I)\big|\,+\, p^2(n-|I|-1)\, .$$
By a calculation similar to~\eqref{syneq1prop} we have
  $$S^p(u,w|I)\,+\, pd(u) \, - \, p^2(n-|I|-1)\,=\, \sum_{v\in V\setminus(I\cup\{u\})}a_v\xi_{wv} \, \,=:\, X_w\, ,$$
  where $a_{v}=1_{v\in N(u)}-p$.  Furthermore, as $\xi_{wv}$ is exactly the indicator function of the edge $wv$, the sequence $(\xi_{wv}:w\in I, v\in V\setminus(I\cup\{u\}))$ is a family of iid $\Bern(p)$ random variables.
  
  As $I$ is simply a set of a fixed size $|I|=k$, and all vertices of $I$ are equivalent, the choice of $w$ and $I$ does not affect the expected value or variance of the random variable $X_w$.  Let us define 
  \[
  \mu_u\, :=\, \Ex{X_w \mid N(u)}\, =\, p\, \sum_{i}a_i\, \, =\, pd(u)\, -\, p^2(n-k-1)
  \]
  and
  \begin{align}
  \sigma^2_u\, :=\, \Var(X_w \mid N(u))\, &=\, p(1-p)\, \sum_{v}a_v^2\nonumber \\ &\hspace{-1mm}=\, p(1-p)^3 d(u)\, +\, p^3(1-p)(n-k-d(u)-1)\, .\hspace{8mm}\label{eq:sigmau}
  \end{align}
  For each $w\in I$ define the renormalised random variable
  \[
  \tilde{X}_w\, :=\, \frac{X_w\, -\, \mu_u}{\sigma_u}\, .
  \]

  We shall now observe that, except with probability at most $\exp(-n^{1/2-2\beta+o(1)})$, each of the following statements hold:
  \begin{enumerate}
      \item[(i)] $|\sigma(p)-\sigma_u|=O(\eps n^{1/4})$,
      \item[(ii)] the vector $\XX=(\tilde{X}_w:w\in I)$ is $(\eps/4)$-close to normal, and
      \item[(iii)] $|S^p(u,w)-S^p(u,w|I)|\, \le \, \eps\sigma(p)/8$ for all $x\in I$.
  \end{enumerate}

We first prove that (i) may only fail if $|D(u)|\ge \eps n^{3/4}$.  This is sufficient, as a standard application of Chernoff's inequality shows that $\pr{|D(u)|\ge \eps n^{3/4}}\le \exp(-n^{1/2-2\beta+o(1)})$.

Assume now that $|D(u)|\le \eps n^{3/4}$.  Let $q:=1-p$, and recall that $k=n^{1/2}\le \eps n^{3/4}$.  By~\eqr{sigmap} and~\eqr{sigmau}, and writing $d(u)=p(n-1)+D(u)$, we have that
 \begin{align*}
 \sigma(p)^2\, -\, \sigma^2_u\, & =\, pq\Big[pq(n-2)\, -\, q^2 d(u)\, -\, p^2(n-k-d(u)-1)\Big]\phantom{\bigg|}\\
 & =\, pq\Big[pq(n-2)\, -\, q^2 (p(n-1)+D(u))\, -\, p^2(qn-k-D(u)-q)\Big]\phantom{\bigg|}\\
 & =\, pq\Big[p^2(k+D(u))\, -\, q^2D(u)\Big]\, +\, O(1)\phantom{\bigg|}\\
 & =\, O(\eps n^{3/4})\, \phantom{\big|} .
 \end{align*}  
 We now see that (i) follows from the difference of squares formula $(\sigma(p)-\sigma_u)(\sigma(p)+\sigma_u)=\sigma(p)^2-\sigma_u^2$, as $\sigma(p)=\Omega(n^{1/2})$.

For statement (ii), simply apply Corollary~\ref{empcor} (with $\eps/12$), to find that the failure probability is at most $24\eps^{-1}\exp(-\eps^2k/288)=\exp(-n^{1/2-2\beta+o(1)})$.

Finally, for (iii), comparing the definitions, we see that
\[
S^p(u,w)-S^p(u,w|I)\, =\, p^2 (k-1)\, -\, p|N(w)\cap I|\, .
\]
As $p(k-1)$ is the mean of the random variable $|N(w)\cap I|\sim  \Bin(k-1,p)$ it is easily verified using Chernoff's inequality that 
\[
\pr{|S^p(u,w)-S^p(u,w|I)|\, >\, \frac{\eps\sigma(p)}{8}}\, \le\, \exp(-\Omega(\eps^2 n/k))\, =\, \exp(-n^{1/2-2\beta+o(1)})\, .
\]
Statement (iii) now follows by a union bound over the $k=\exp(n^{o(1)})$ vertices $w\in I$.	

It now clear suffices to prove that, if (i), (ii) and (iii) holds then the vector~\eqr{sampleS} is $(\eps/2)$-close to normal.  From now we assume (i), (ii) and (iii) hold.

\textbf{Claim} If $|\tilde{X}_w|\le n^{1/8}$ then
\[
\big| \tilde{S}^p(u,w)\, -\, \tilde{X}_w\big|\, \le\, \frac{\eps}{4}
\]
provided $n$ is sufficiently large.

\textbf{Proof of Claim:} We begin by noting that $X_w-\mu_u=S^{p}(u,w|I)$, so that $\tilde{X}_w=S^{p}(u,w|I)/\sigma_u$.  By (i), we have
\begin{align*}
\tilde{S}^p(u,w)\, -\, \tilde{X}_w\, &=\, \frac{S^p(u,w)}{\sigma(p)}\, -\, \frac{S^p(u,w|I)}{\sigma_u}\\
& =\, \frac{S^p(u,w)-S^p(u,w|I)}{\sigma(p)}\, \pm\, O(\eps n^{-1/4} \tilde{X}_w)\, .
\end{align*}
Using the assumption on $\tilde{X}_w$, and (iii), we obtain that $|\tilde{S}^p(u,w)\, -\, \tilde{X}_w|$ is at most $\eps/8+O(\eps n^{-1/8})$, which is at most $\eps/4$, for all sufficiently large $n$, as required.

We are now finally ready to observe that~\eqr{sampleS} is $(\eps/2)$-close to normal.  In other words, we prove that
\[
\forall a\in \RR \qquad \Big|\, |\{w\in I:\tilde{S}^p(u,w)\le a\}|\, -\, \Phi(a)k\, \Big|\, \le\, \frac{\eps k}{2}\, .
\]
For values $a\in \RR$ with $|a|\ge n^{1/8}-\eps$ it suffices to show that at most $\eps k/2$ vertices $w\in I$ have $|\tilde{S}^p(u,w)|\ge n^{1/8}-\eps$.  It follows from (ii) that, for all sufficiently large $n$, there are at most $\eps k/2$ vertices $w\in I$ such that $|\tilde{X}_w|\ge n^{1/8}-2\eps$.  For all the other vertices $w\in I$ we have $|\tilde{S}^p(u,w)|\le n^{1/8}-\eps$ by the claim.

For $a\in \RR$ with $|a|\le n^{1/8}-\eps$, we use (ii) and the claim to obtain that
\[
\frac{|\{w\in I:\tilde{S}^p(u,w)\le a\}|}{k}\, \le\, \frac{|\{w\in I:\tilde{X}_w\le a+\eps/4\}|}{k}\, \le\, \Phi\left(a+\frac{\eps}{4}\right)+\frac{\eps}{4}\, \le\, \Phi(a)+\frac{\eps}{2}\, ,
\]
and
\[
\frac{|\{w\in I:\tilde{S}^p(u,w)\le a\}|}{k}\, \ge\, \frac{|\{w\in I:\tilde{X}_w\ge a-\eps/4\}|}{k}\, \ge\, \Phi\left(a-\frac{\eps}{4}\right)-\frac{\eps}{4}\, \ge\, \Phi(a)-\frac{\eps}{2}\, ,
\]
as required.
\end{proof}

		
		
		
		

\section{Relating Codegrees and Synergies}\label{sec:CoSy}

As we explained in the overview, our plan is to consider an event which biases our selection towards edges with smaller synergy.   The results of this section relate synergies to codegrees, which is vital to our approach.  We also prove some properties of the graph $F_-$.


More concretely, we shall use $\alpha\in(0,1)$ as a parameter to be defined later. Our event will involve taking $(1+\alpha)/2$ proportion of the remaining $m_1=m-m_0$ edges from $F_-$ rather than its complement $F_+:= K_n\setminus (G^0\cup F_-)$.  One contribution to the final number of triangles are those with two edges in $G^0$ and one in $G^1$.  Conditioned on this event this conditional expectation becomes
$$\sum_{uv\in F_-}d_{G^0}(u,v)\frac{(1+\alpha)m_1}{N-m_0}+\sum_{uv\in F_+}d_{G^0}(u,v)\frac{(1-\alpha)m_1}{N-m_0}\, .$$
This clearly motivates our interest in the expression
\[
\sum_{uv\in F_-}d_{G^0}(u,v)\, -\,\sum_{uv\in F_+}d_{G^0}(u,v)\, .
\]


We now state the required result on this quantity.  Let us fix for the remainder of the article $\varepsilon:=n^{-1/5}$ and recall that $p_0=m_0/N$ and 
$\sigma(p_0)^2=\Var(S_{G}^{p_0}(u,v))$, where $G\sim G(n,p_0)$, which we abbreviate to $\sigma^2$ for the remainder of the section.


\begin{prop}\label{prop: main CoSy}
There exists $c>0$ such that the following holds.  If $G^0$ is such that $|d_{G^0}(u)-p_0n|\leq n^{1/2}\log n$ and $\mathcal{S}_{G^0}(u)$ is $\varepsilon$-close to normal, for all vertices $u$, then 
    $$\sum_{uv\in F_-}d_{G^0}(u,v)-\sum_{uv\in F_+}d_{G^0}(u,v)\leq -cn^{5/2}.$$
\end{prop}

Assume, for the rest of this section, that for every $u\in V(G^0)$ we have $\mathcal{S}_{G^0}(u)$ is $\varepsilon$-close to normal. The next lemma relates this quantity to an equivalent expression using codegrees, up to an error term.

\begin{lem}\label{d-to-syn}
	Let $D=max_{u}\{|d_{G^0}(u)-p_0n|:u\in V(G^0)\}$. Then
	\begin{equation*}
		\begin{split}
			\sum_{uv\in F_-}d_{G^0}(u,v)-\sum_{uv\in F_+}d_{G^0}(u,v)&\, =\, \sum_{uv\in F_-}S_{G^0}(u,v)-\sum_{uv\in F_+}S_{G^0}(u,v)\pm 4p_0 \eps Dn^2.\\
		\end{split}
	\end{equation*}
\end{lem}

We first require a lemma about the degrees of vertices in $F_-$.

\begin{lem}\label{lem:1-(2,1)} 
For every vertex $u$ we have
$$d_{F_{-}}(u)\, =\, \left(\frac{1}{2}\pm 2\eps\right)(n-d_{G^0}(u)-1)\,.$$
\end{lem}

  \begin{proof}
      We begin by noting that the normal density $\phi(t)$ satisfies $\phi(t)\ge 1/3$ for all $|t|\le 1/10$.  So that, in particular, $\Phi(3\eps)\ge 1/2 +\eps$.  Let $\lambda>0$ be such that $\Phi(\lambda)=1/2+\eps$.  We have $\lambda\le 3\eps$.

Recall that $F_-$ consists of the pairs of $K_n\setminus G^0$ with synergy $S_{G^0}(u,v)$ below the median.  Let $\tilde{\mu}$ be the value of this median.

As each $\mathcal{S}_{G^0}(u)$ is $\varepsilon$-close to normal, we have for all vertices $u$, that at least
\[
(\Phi(\lambda)-\eps)(n-d_{G^0}(u)-1)\, =\, \frac{1}{2}(n-d_{G^0}(u)-1)
\]
pairs $uv\in E(K_n\setminus G^0)$ have $S_{G^0}(u,v)\le \lambda\sigma$.  Summing over vertices $u$, it follows that at least $(N-m_0)/2$ pairs $uv\in E(K_n\setminus G^0)$ have $S_{G^0}(u,v)\le \lambda\sigma$.  In particular, $\tilde{\mu}\le \lambda\sigma$.  And, by a similar argument, $\tilde{\mu}\ge -\lambda\sigma$. 

In particular, by monotonicity $\Phi(\tilde{\mu})=1/2\pm \eps$.


As $\mathcal{S}_{G^0}(u)$ is $\varepsilon$-close to normal we have
\[
d_{F_-}(u)\, =\, (\Phi(\tilde{\mu})\pm \eps) (n-d_{G^0}(u)-1)\, =\, \left(\frac{1}{2}\pm 2\eps\right)(n-d_{G^0}(u)-1)\, .
\]
  \end{proof}

We may now prove the previous lemma.

\begin{proof}[Proof of Lemma~\ref{d-to-syn}]
Our objective is to bound the absolute value of the expression
\eq{remarksyn}
\sum_{uv\in F_-}d_{G^0}(u,v)\, -\, \sum_{uv\in F_+}d_{G^0}(u,v)\, -\, \sum_{uv\in F_-}S_{G^0}(u,v)\, +\, \sum_{uv\in F_+}S_{G^0}(u,v)\, .
\eqe
By the definition of synergy, and as $|F_-|=|F_+|$, this expression is precisely
	\begin{equation}
		\sum_{uv\in F_-}p_0d_{G^0}(u)+p_0d_{G^0}(v)-\sum_{uv\in F_+}p_0d_{G^0}(u)+p_0d_{G^0}(v).
	\end{equation}
As $d_{G^0}(u)$ appears $d_{F_-}(u)$ times in the first sum and $d_{F_+}(u)$ times in the second sum, we may further simplify the expression~\eqr{remarksyn} to
	\[
			p_0\, \sum_{u\in V(G)} d_{G^0}(u)\big(d_{F_-}(u)\, -\, d_{F_+}(u)\big).	
   \]
   As $|F_-|=|F_+|$ we have $\sum_{u\in V(G)}d_{F_-}(u)-d_{F_+}(u)=0$ and hence subtracting the constant term $p_0n$ from every degree $d_{G^0}(u)$, we see that~\eqr{remarksyn} is precisely
   \[
			p_0\, \sum_{u\in V(G)} \big(d_{G^0}(u)-p_0n\big)\big(d_{F_-}(u)\, -\, d_{F_+}(u)\big)\, .
   \] 
  This is at most $4p_0 \eps Dn^2$ by Lemma~\ref{lem:1-(2,1)}, as required.
   \end{proof}


We now proceed to compute the value of $\sum_{uv\in F_+}S_{G^0}(u,v)$. The computation for $\sum_{uv\in F_-}S_{G^0}(u,v)$ is almost identical and is omitted. First we introduce the set $F_{+,0}=\{uw\in E(K_n\setminus G^0):S_{G^0}(u,w)\geq 0\}$. The introduction of this set is useful since $\sum_{uv\in F_{+,0}}S_{G^0}(u,v)$ is easier to compute and is further justified by the following lemma.


\begin{lem}\label{lem2-(2,1)}
	$$\left|\sum_{F_{+,0}}S_{G^0}(u,v)-\sum_{F_{+}}S_{G^0}(u,v)\right|\, \leq\, 9\varepsilon^2(N-m_0)\sigma.$$
\end{lem}
\begin{proof}
As $\Phi(0)=1/2$ and the synergies are close to normal, for every vertex $u$ we have  that
\eq{F+0reg}
d_{F_{+,0}}(u)\, =\, \left(\frac{1}{2}\pm \eps\right)(n-d_{G^0}(u)-1)\, .
\eqe	

We saw in the proof of Lemma~\ref{lem:1-(2,1)} that the median of the synergies $\tilde{\mu}$ satisfies $|\tilde{\mu}|\le 3\eps$.  
There are two cases depending on whether $\tilde{\mu}>0$ or not.  We consider the case $\tilde{\mu}>0$, the other case is essentially identical.

As $\tilde{\mu}>0$, $F_+$ is contained in $F_{+,0}$. By~\eqr{F+0reg} and Lemma~\ref{lem:1-(2,1)} respectively we have 
\[
e(F_{+,0})\, \le\, \left(\frac{1}{2}+\eps\right)(N-m_0)\qquad \text{and} \qquad e(F_{+})\, \ge\, \left(\frac{1}{2}-2\eps\right)(N-m_0) \, .
\]
Therefore there are at most $3\eps (N-m_0)$ \emph{extra} edges in $F_{+,0}$.  Furthermore, all such edges must have 
\[
0\, \le\,  S_{G^0}(u,v)\, \le\, \tilde{\mu}\sigma\, \le\, 3\eps\sigma\, .
\]
It follows that
\[
\left|\sum_{F_{+,0}}S_{G^0}(u,v)-\sum_{F_{+}}S_{G^0}(u,v)\right|\, \le\, 3\eps(N-m_0)\, \cdot\, (3\eps\sigma)\, =\, 9\eps^2(N-m_0)\sigma\, ,
\]
as required.\end{proof}

With the last lemma it only remains to prove a lower bound for the value of $\sum_{F_{+,0}}S_{G^0}(u,v)$.
\begin{lem}\label{lem:3-(2,1)}
	$$\sum_{uv\in F_{+,0}} S_{G^0}(u,v)\geq \frac{(1+o(1))\sigma}{\sqrt{2\pi}}(N-m_0).$$
\end{lem}
\begin{proof}
    Let $S$ be such that $\Phi(S)=1-\eps$ and $Z\sim N(0,1)$. Then, since $1-\Phi(x)=\mathbb{P}(Z>x)$ and the standard normal distribution has the tail inequality $\mathbb{P}(Z>x)\leq e^{-x^2/2}$,  we have
    $$S\leq \sqrt{2\log\varepsilon^{-1}}.$$
 We remark that standard estimates show that, for $x>0$, $\int_{S}^{\infty}\pr{Z>x} dx\le 2\eps S $, and recall that $\int_{0}^{\infty}\pr{Z>x}dx\, =\, 1/\sqrt{2\pi}$.  We shall combine these to obtain $\int_{0}^{S}\pr{Z>x}dx\, \ge\, 1/\sqrt{2\pi}\, -\, 2\eps S$.
 
  Since for every $u\in V(G^0)$ the sequence $\mathcal{S}_{G^0}(u)$ is $\varepsilon$-close to normal, we have, for $v\not\in N_{G^0}(u)$,
 $$\big||\{v:S_{G^0}(u,v)>x\sigma\}|-(1-\Phi(x))(n-d_{G^0}(u)-1)\big|\, \le\, \varepsilon (n-d_{G^0}(u)-1).$$
 Hence,
	\begin{equation*}
		\begin{split}
			\sum_{uv\in F_{+,0}} S_{G^0}(u,v)\,&=\,\frac{\sigma}{2} \sum_{u\in V(G)}\sum_{v\in N_{F_{+,0}}(u)}\frac{S_{G^0}(u,v)}{\sigma}\\
			&=\, \frac{\sigma}{2} \sum_{u\in V(G)}\int_0^\infty [\#v:S_{G^0}(u,v)>x\sigma]dx\\
			&\geq\, \frac{\sigma}{2}  \sum_{u\in V(G)} (n-d_{G^0}(u)-1)\int_0^S (1-\Phi(x))-\varepsilon \,dx\\
                &\geq\, \sigma(N-m_0)\left(\frac{1}{\sqrt{2\pi}}-3\varepsilon S\right)
		\end{split}
	\end{equation*}
	 As $S\leq\sqrt{2\log\varepsilon^{-1}}$, then 
the error term is $o(1)$, as required.

\end{proof}

We can finally prove Proposition \ref{prop: main CoSy}. Lemma \ref{d-to-syn} shows that is enough to bound $$\sum_{uv\in F_-}S_{G^0}(u,v)-\sum_{uv\in F_+}S_{G^0}(u,v)\pm \varepsilon p_0Dn^{2}.$$
Lemmas \ref{lem2-(2,1)} and \ref{lem:3-(2,1)} let us compute
\begin{equation*}
    \begin{split}     
    \sum_{uv\in F_+}S_{G^0}(u,v)\,&\geq\, \sum_{uv\in F_{+,0}}S_{G^0}(u,v)-3\varepsilon^2(N-m_0)\sigma\\
    &\geq\, \frac{\sigma}{\sqrt{2\pi}}(N-m_0)-3\varepsilon^2(N-m_0)\sigma\\
    &=\,\left(\frac{1}{\sqrt{2\pi}}-o(1)\right)\sigma(N-m_0)
    \end{split}
\end{equation*}
Futhermore, $D\leq n^{1/2}\log n$ by hyphoteses. Hence, $\varepsilon p_0 Dn^2=o(n^{5/2})$, as $\varepsilon=n^{-1/5}$.

Finally, as the sum of the synergies of $F_-$ has the same results and $\sigma=\Theta(n^{1/2})$, see \ref{claim:sigmap}, there is a $c>0$ such that
$$\sum_{uv\in F_-}S_{G^0}(u,v)-\sum_{uv\in F_+}S_{G^0}(u,v)\pm \varepsilon p_0Dn^{2}\leq -cn^{5/2}$$
as claimed.

\section{Defining the event $E_{\alpha}$}\label{sec:event}

In this section we define events $E_{\alpha}:\alpha>0$, which shall be used to prove Proposition~\ref{prop:main}.  (The value of $\alpha$ will be chosen appropriately as a function of $\delta$.)

Recall that $\eps:=n^{-1/5}$.  Let $\eta:=\eta(n)>0$ be a constant which we choose (sufficiently small) later.  We shall assume that $N-(1-\eta) m$ is an even integer\footnote{This is a minor abuse of notation.  Only the order of magnitude of $\eta$ matters for our purposes, and this error is of the form $\pm O(n^{-2}$) with respect to the size of $\eta$.}  for all sufficiently large $n$.

Define $m_0:=(1-\eta)m$ and $m_1=\eta m$.  As mentioned in Section~\ref{sec:overview} we may generate $G_{(m)}$ using the Erd\H{o}s-R\'enyi random graph process.  In a slight abuse of notation let $G^0:=G_{(m_0)}\sim G(n,m_0)$ and let $G^1:=G_{(m)}\setminus G_{(m_0)}\sim G(n,m_1)$.  That is, $G^0$ represents the first $m_0$ edges of our random graph $G_{(m)}$, and $G^1$ represents the remaining $m_1=m-m_0$ edges.

Before defining $E_{\alpha}$, let us mention some useful quasirandomness properties for the first graph $G^0\sim G(n,m_0)$.   

By Lemma 4.6 of~\cite{GGS}, see also Corollary 5.2 of~\cite{AGG}, applied to $(G^{0})^c\sim G(n,N-m_0)$, there exists an absolute constant $C_D$ such that
\eq{P1} \tag{$P_1$}
\sum_{u,w\in V((G^{0})^c)}D_{(G^{0})^c}(u,w)^2\, \le\, C_D n^3\, 
\eqe
with high probability.  We also have, by Proposition~\ref{prop:syngm}, that with high probability
\eq{P2}\tag{$P_2$}
\forall u \in V(G^0)\quad \S_{G^0}(u)\, \, \text{is } \eps\text{-close to normal}\, .
\eqe
The following properties concern degrees, codegrees, edge counts and triangle counts.  We omit the quantifications.  It should be understood that the properties are for all vertices $u$, pairs of vertices $u,w$, sets $U$, and disjoint pairs of sets $U,V$ respectively.  Recall that $p_0=m_0/N$.
\eq{P3}\tag{$P_3$}
\big|d_{G^0}(u)\, -\,  p_0n\big|\, \le\, n^{1/2}\log{n}\phantom{\bigg|}
\eqe
\eq{P4}\tag{$P_4$}
\big|d_{G^0}(u,w)\, -\,  p_0^2n\big|\, \le\, n^{1/2}\log{n}\phantom{\bigg|}
\eqe
\eq{P5}\tag{$P_5$}
\left|e(G^0[U])\, -\,  p_0\binom{|U|}{2}\right|\, \le\, n^{3/2}\phantom{\Bigg|}
\eqe
\eq{P6}\tag{$P_6$}
\big|e(G^0[U,V])\, -\,  p_0|U||V|\big|\, \le\, n^{3/2}\phantom{\bigg|}
\eqe

Let $E_0$ be the event that all six properties ($P_1$)-($P_6$) are satisfied by the graph $G^0$.  We shall in fact prove that $E_0$ has very small failure probability, and this will be useful to bound the effect of conditioning on $E_0$. 
 
\begin{lem}\label{lem:likely}
	$$\mathbb{P}(E_0)\,\geq\, 1-n^{-3}.$$
\end{lem}
\begin{proof}
    Let's compute the probability of each property not happening. We may already bound the probability that properties \eqr{P1} or \eqr{P2} do not hold.  Lemma 4.6 of~\cite{GGS} and Proposition \ref{prop:syngm} shows that, respectively,
    $$\mathbb{P}(P_1^c)\leq e^{-n}\qquad\text{and}\qquad  \mathbb{P}(P_2^c)\leq \exp(-n^{1/10+o(1)}).$$

    The bounds for the probabilities of \eqr{P3}-\eqr{P6} are all given by the Chernoff bound for hypergeometric random variables.  For example, by Theorem 2.10 of~\cite{JLR}, in particular (2.5) and (2.6), we have that, for all $t\le \mu$,
    \[
    \pr{|X-\mu|\, \ge\, t}\, \le\, 2\exp\left(\frac{-t^2}{2(\mu+t/3)}\right)\, \le\, 2\exp\left(\frac{-3t^2}{8\mu}\right)
    \]
    where $X$ is a hypergeometric random variable with mean $\mu$.
    Applying this inequality, and appropriate union bounds we obtain
    $$\mathbb{P}(P_3^c),\mathbb{P}(P_4^c)\,\leq\, n^2e^{-(\log n)^2} \qquad\text{and}\qquad
    \mathbb{P}(P_5^c),\mathbb{P}(P_6^c)\,\leq\, e^{-cn}\, ,$$
for some constant $c>0$.  For example, in the case of \eqr{P6}, the Chernoff bound gives a bound of at most $2\exp(-3(n^{3/2})^2/8(p_0 n^2/4))\le 2\exp(-3n/2)$.  And so, even taking a union bound over at most $3^n$ choices of the sets $U,V$, we obtain a bound of the form $e^{-cn}$, as required.
    
    As each of these bounds is smaller then $n^{-4}$, for all sufficiently large $n$, it follows that
    $P(E_0^c)\leq n^{-3}$,
    as claimed
\end{proof}

We now deduce that conditioning on $E_0$ has little effect on expected values of functions $f:\mathcal{G}_{n,m}\to\mathbb{R}$, where $\mathcal{G}_{n,m}$ is the set of graphs with $n$ vertices and $m$ edges.

\begin{lem}\label{lem:E0 good}
    Given a function $f:\mathcal{G}_{n,m}\to\mathbb{R}$, let $\overline{F}:=\max\{f(G):G\in\mathcal{G}_{n,m}\}$ and $\underline{F}:=\min\{f(G):G\in\mathcal{G}_{n,m}\}$. Then $$\Big|\mathbb{E}\big[f(G_{(m)})\,|\,E_0\big]\,-\, \mathbb{E}\big[f(G_{(m)})\big]\Big| \, \le\,  (\overline{F}-\underline{F})n^{-3}.$$
\end{lem}
\begin{proof} By conditioning, we have
    \begin{equation*}
        \begin{split}
            \mathbb{E}\big[f(G_{(m)})\big]\,&=\, \mathbb{P}(E_0)\mathbb{E}\big[f(G_{(m)})|E_0\big]+\mathbb{P}(E_0^c)\mathbb{E}\big[f(G_{(m)})|E_0^c\big]\\
            &\,=\mathbb{E}\big[f(G_{(m)})|E_0\big]+\mathbb{P}(E_0^c)\big(\mathbb{E}\big[f(G_{(m)})|E_0^c\big]-\mathbb{E}\big[f(G_{(m)})|E_0\big]\big)\, .
        \end{split}
    \end{equation*}
    The claimed bound now follows, as $\mathbb{P}(E_0^c)\le n^{-3}$, by Lemma \ref{lem:likely}, and the difference between the conditional expectations has absolute value at most $(\overline{F}-\underline{F})$. 
\end{proof}

We shall soon be ready to define the event $E_{\alpha}$.  Let us first recall (from Section \ref{sec:overview}) that we let $f_1,\dots f_{N-m_0}$ be the non-edges of $G^0$ in non-decreasing order of synergy $S(u,w)$, and define
\[
F_-\, :=\,\left\{f_i:i\in\left\{1,\dots,\frac{N-m_0}{2}\right\}\right\}\, ,
\]
that is $F_-$ is a graph consisting of the half of non-edges of $G^0$ with smaller synergy.

\begin{definition} Let $E(\alpha)$ be the event\footnote{We again commit a slight abuse in assuming that $(1+\alpha)m_1/2$ is an integer.  This is a minor abuse of notation, as only the order of magnitude of the deviation really matters for our argument.} that $|F_-\cap E(G^1)|=(1+\alpha)m_1/2$. 
\end{definition}

\begin{definition}
Let $E_\alpha:=E_0\cap E(\alpha)$.
\end{definition}

In other words, $E(\alpha)$ is the event that $G^1$ takes significantly more of its edges in $F_-$ (the pairs of low synergy) than expected.  Hence $E_\alpha$ is simply the event that $G^0$ is well behaved (satisfying \eqr{P1}-\eqr{P6}) \emph{and} this event $E(\alpha)$ occurs.

While $G^0$ determines which edges are in $F_-$, this does not affect the probability of $E(\alpha)$, and so $E(\alpha)$ is independent of $G^0$.  The conditional expectation with respect to $E_{\alpha}$ may be calculated as
\eq{conditionE}\tag{5.1}
\Ex{f(G)|E_{\alpha}}\, =\, \sum_{G'_0}\pr{G^0=G'_0|E_0}\Ex{f(G)|E(\alpha),G^0=G'_0}\, .
\eqe

We now proceed to prove a lower bound on the probability of $E_\alpha$.

\begin{prop}\label{prop:cost} Let $\eta,\lambda\in (0,1)$ and suppose that $\alpha=\alpha_n$ is a sequence such that $\log{n}/n\ll \alpha\ll 1$. Then
	\[
		\pr{E_\alpha} \,\geq\, \expb{-\lambda^{-1}\alpha^2 n^2}
	\]
 for all sufficiently large $n$.
\end{prop}

\begin{proof} As $E_0$ and $E(\alpha)$ are independent, and $\pr{E_0}=1+o(1)$ by Lemma~\ref{lem:likely}, it suffices to prove that the claimed bound holds for the event $E(\alpha)$.  The event $E(\alpha)$ corresponds to a random variable with hypergeometric distribution $\textrm{Hyp}\big(N-m_0, (N-m_0)/2,m_1\big)$ taking the value $(1+\alpha)\mu$.  In the appendix we prove (Corollary~\ref{cor:hyper}) that this probability is at least $\exp(-\lambda^{-1}\alpha^2 (N-m_0))$.  The required bound easily follows.
\end{proof}

We remark that conditioning on $E_0$ (or $E_{\alpha}$) also gives us the following property on the degrees of vertices in $F_-$:
\eq{F-degs}\tag{5.2}
d_{F_{-}}(u)\, =\, \left(\frac{1}{2}\pm 2\eps\right)(n-d_{G^0}(u)-1)\, .
\eqe
Indeed Lemma~\ref{lem:1-(2,1)} shows that this is a consequence of property~\eqr{P2}.
As $d_{F-}(u)+d_{F+}(u)=n-d_{G^0}(u)-1$, the same result follows for $F_+$.


\section{Completing the proof of Theorem~\ref{thm:main}}\label{sec:proof}

As outlined in Section~\ref{sec:overview}, it suffices to prove Proposition~\ref{prop:main}, which we now restate.

\propmain*

Fix a value of $\lambda\in (0,1)$.  As we have discussed, the event $E$ we shall use is one of the events $E_{\alpha}$ defined in the previous section.  We shall take $\alpha$ to be $\alpha=C' \delta n^{1/2}$ for a constant $C'$ below.

By Proposition~\ref{prop:cost}, it is clear that condition (i) of Proposition~\ref{prop:main} holds for the event $E_{\alpha}$, as long as we take $C\ge \lambda^{-1}(C')^2$.  All that remains is to prove that
\[
\Ex{N_{\triangle}(G_{(m)})\, |\, E_{\alpha}}\, \le\, (1-2\delta)\mu_{n,m}\, ,
\]
and it will be done in this section.

		We consider the triangles in four classes: those with three edges from $G^0$, two edges from $G^0$ and one from $G^1$, one from $G^0$ and two from $G^1$ or three from $G^1$. We will denote by $\triangle_{(3,0)}$, $\triangle_{(2,1)}$, $\triangle_{(1,2)}$ and $\triangle_{(0,3)}$, 
the respective counts.

Let $q_0:=1-p_0$ and $p_1:=\mathbb{P}\big(e\in E(G^1)\big)=\eta m/N$.  Note that $p=p_0+p_1$. We shall deduce (ii) from the following four lemmas.
\begin{lem}\label{lem:3,0}
$$\mathbb{E}\left[\triangle_{(3,0)}|E_\alpha\right]\,=\,p_0^3\binom{n}{3}\, \pm\,  1.$$
\end{lem}
\begin{lem}\label{lem:2,1}
    There is a constant $c_{(2,1)}>0$ such that
$$\mathbb{E}\left[\triangle_{(2,1)}|E_\alpha\right]\,\leq\, 3p_0^2p_1\binom{n}{3}\, -\, c_{(2,1)}\, \eta \alpha  n^{5/2}\, \pm\, 1.$$
\end{lem}
\begin{lem}\label{lem:1,2}
    There is a constant $c_{(1,2)}>0$ such that      $$\mathbb{E}\left[\triangle_{(1,2)}|E_\alpha\right]\,\leq\, 3p_0p_1^2\binom{n}{3}\, +\, c_{(1,2)}\,\eta^2 \alpha n^{5/2}\,\pm\, 1.$$
\end{lem}
\begin{lem}\label{lem:0,3}
    There is a constant $c_{(0,3)}>0$ such that $$\mathbb{E}\left[\triangle_{(0,3)}|E_\alpha\right]\,\leq\, p_1^3\binom{n}{3}\, +\, c_{(0,3)}\,\eta^3 \alpha  n^{5/2}\pm 1.$$	
\end{lem}

We now prove that (ii) follows from the above lemmas.  Given the constants  $c_{(2,1)}$, $c_{(1,2)}$ and $c_{(0,3)}$ of the lemmas, we now fix the value of $\eta>0$ which we shall use.

As $\eta^2 c_{(0,3)}\,+\,\eta c_{(1,2)}\,- \, c_{(2,1)}$ is increasing (for $\eta\ge 0$), we may choose $\eta>0$ such that $\eta^2 c_{(0,3)}\,+\,\eta c_{(1,2)}\,- \, c_{(2,1)}\, =\, -c_{(2,1)}/2$.  Let us set $C'=1/c_{(2,1)}\eta$.

\begin{proof}[Proof of (ii)]
As $\mathbb{E}[N_\triangle(G_{(m)})]=\mu_{n,m}=p^3\binom{n}{3}=(p_0+p_1)^3\binom{n}{3}$, the main terms sum to $\mu_{n,m}$.  And so,
\begin{equation*}
\begin{split}
\mathbb{E}[N_\triangle(G_{(m)})\,|\, E_\alpha]\,&=\, \mathbb{E}\left[\triangle_{(3,0)}|E_\alpha\right]+\mathbb{E}\left[\triangle_{(2,1)}|E_\alpha\right]+\mathbb{E}\left[\triangle_{(1,2)}|E_\alpha\right]+\mathbb{E}\left[\triangle_{(0,3)}|E_\alpha\right]\phantom{\Big|}\\
            &\leq\, \mu_{n,m}\, +\, \eta\alpha n^{5/2}\big(\eta^2 c_{(0,3)}\,+\,\eta c_{(1,2)}\,- \, c_{(2,1)}\big)\, \pm\,  4\phantom{\Big|}\\
& =\,         \mu_{n,m}\, -\, \frac{c_{(2,1)}\eta\alpha n^{5/2}}{2} \, \pm \, 4
       \end{split}
    \end{equation*}
Recalling that we shall choose $\alpha=C'\delta n^{1/2}$, where $C'= 1/c_{(2,1)}\eta$, we obtain 
\[
\mathbb{E}[N_\triangle(G_{(m)})\,|\, E_\alpha]\, \le\, \mu_{n,m}\, -\, \frac{\delta n^3}{2}\, +\, 4\, \le\, \mu_{n,m}(1-2\delta).
\]
\end{proof}

A remark on the choice of constants:
\begin{itemize}
\item The constants $c_{(2,1)}, c_{(1,2)}$ and $c_{(0,3)}$ which occur in the lemmas may depend only on $\lambda$.
\item We chose $\eta$ as a function of $c_{(2,1)}, c_{(1,2)}$ and $c_{(0,3)}$.
\item We chose $C'= 1/c_{(2,1)}\eta$ as a function of $c_{(2,1)}$ and $\eta$.
\item The final constant $C$ of Proposition~\ref{prop:main} may be taken to be $C=\max\{C',\lambda^{-1}(C')^2\}$.
\end{itemize}

In order to complete the proof, we prove the lemmas above in the following subsections.  As $C\ge C'$, we note in particular that
\eq{alphale}
\alpha\, =\, C'\delta n^{1/2}\, \le\, \frac{C'}{C} n^{-3/4} n^{1/2}\, \le\, n^{-1/4}\, .
\eqe


\subsection{Proof of Lemma \ref{lem:3,0}}		
		The first type of triangle, $\triangle_{(3,0)}$, is independent of $G^1$, and so
		\begin{equation}\label{eq:triangle(3,0)}
    \mathbb{E}\big[\triangle_{(3,0)}\,|\,E_\alpha\big]\, =\, \mathbb{E}\big[\triangle_{(3,0)}\,|\,E_0\big]\, =\,p_0^3\binom{n}{3}\, \pm\, 1.
		\end{equation}
  using Lemma \ref{lem:E0 good} with $\max\{\triangle_{(3,0)}\}\leq n^3$ and $\min\{\triangle_{(3,0)}\}\ge 0$.
\subsection{Proof of Lemma \ref{lem:2,1}}
	
We recall that our whole strategy for creating fewer triangles relies on creating fewer triangles of this type.  We now show how this works.
 
 Recall that $E_\alpha=E_0\cap E(\alpha)$, that is, $E_\alpha$ is the intersection of $E_0$, the event that $G^0$ satisfies the properties \eqr{P1}-\eqr{P6}, and $E(\alpha)$, the event that (given the graph $F_-$ defined by $G^0$) $G^1$ contains $(1+\alpha)m_1/2$ edges of $F_-$.
 In particular, given $G^0$, on the event $E_{\alpha}$ each edge of $F_-$ is included in $G^1$ with probability $(1+\alpha)m_1/(N-m_0)$, and each edge of $F_+$ with probability $(1-\alpha)m_1/(N-m_0)$.
 
 For convenience, we use the notation $d(u,v)$ for $d_{G^0}(u,v)$, the codegree of $u$ and $v$ in $G^0$, and $d(u)$ for $d_{G^0}(u)$, the degree of $u$ in $G^0$.  We have
	\begin{equation*}
		\begin{split}
			\mathbb{E}\left[\triangle_{(2,1)}|E_\alpha\right]& =\mathbb{E}\left[\sum_{uv\in F_-}\frac{(1+\alpha)m_1}{N-m_0}d(u,v)\, +\, \sum_{uv\in F_+}\frac{(1-\alpha)m_1}{N-m_0}d(u,v)\,\middle| \,E_0\right] \phantom{\scalebox{1.5}{\Bigg|}}\\
			& \hspace{-25mm} =\, \mathbb{E}\left[\triangle_{(2,1)}\,\middle|\, E_0\right]\, +\,  \frac{p\eta\alpha}{1-p_0}\, \mathbb{E}\left[\sum_{uv\in F_-}d(u,v)-\sum_{uv\in F_+}d(u,v)\,\middle|\, E_0\right]\, ,\phantom{\scalebox{1.5}{\Bigg|}}
		\end{split}
	\end{equation*}
 where the second equality uses that the main term produces exactly the expected value $\mathbb{E}[\triangle_{(2,1)}\,|\, E_0]$, and that $m_1=\eta pN$.
	
 As we are conditioning on $E_0$, we know that properties \eqr{P2} and \eqr{P3} hold for $G^0$, and so by Proposition \ref{prop: main CoSy} we have
 \[
 \sum_{uv\in F_-}d(u,v)\, -\, \sum_{uv\in F_+}d(u,v)\, \leq\, -cn^{5/2}
 \]
 for some constant $c>0$, on the event $E_0$.  Taking $c_{(2,1)}=cp$, and using Lemma \ref{lem:E0 good}, we now have
	\begin{equation*}
		\begin{split} \mathbb{E}\left[\triangle_{(2,1)}|E_\alpha\right]\, &\leq \, 3p_0^2p_1\binom{n}{3}\, -\, c_{(2,1)}\, \eta \alpha  n^{5/2}\, .
		\end{split}
	\end{equation*}

\subsection{Proof of Lemma \ref{lem:1,2}}
We now proceed to the third type of triangles, $\triangle_{(1,2)}$.  A triangle of type $(1,2)$ consists of vertices $u,v,w$, such that $vw$ is an edge of $G^0$ and both $uv$ and $uw$ are edges of $G^1$.  So that
\[
\triangle_{(1,2)}\, =\, \sum_{u,v,w}1_{vw\in E(G^0)}1_{uv,uw\in E(G^1)}\, .
\]
Let $e(u):=e(G^0[V\setminus (N(u)\cup \{u\})])$ denote the number of edges of $G^0$ amongst the non-neighbours of $u$.  Each edge of $E(K_n)\setminus E(G^0)$ is in $G^1$ with probability $r_1:=p_1/(1-p_0)$, and each such pair of edges is in $G^1$ with probability $r_1^2 +O(n^{-2})$.  It follows that
\[
\mathbb{E}[\triangle_{(1,2)}|G^0]\, =\, r_1^2 \sum_{u} e(u)\, \pm\, O(n)\, .
\]
If we condition only on the event $E_0$  then, as $r_1(1-p_0)=p_1$, and using Lemma~\ref{lem:E0 good}, we have
\begin{align*}
\Ex{\triangle_{(1,2)}|E_0}\, &=\, r_1^2 \sum_{u} \Ex{e(u)|E_0}\, \pm\, O(n)\\
&=\,  r_1^2 \sum_{u} \Ex{e(u)}\, \pm\, O(n)\\ 
&=\, 3p_0p_1^2\binom{n}{3}\, \pm\, O(n)\, .
\end{align*}
What changes when we condition on $E_{\alpha}$?  As $E_{\alpha}=E_0\cap E(\alpha)$, for a graph $G^0$ satisfying the conditions of $E_0$, the conditioning affects the probabilities that the edges $uv$ and $uw$ are included in $G^1$.  As they will now depend on whether $uv$ and $uw$ belong to $F_-$ or $F_+$.  If both belong to $F_-$ then the probability $r_1^2 +O(n^{-2})$ is multiplied by a factor of $(1+\alpha)^2$.  If one is in $F_-$ and the other in $F_+$ then the factor is $(1+
\alpha)(1-\alpha)$.  Finally, if both are in $F_+$ then the factor is $(1-\alpha)^2$.

Let us define $e_-(u):=e(G^0[N_{F_-}(u)])$ to be the number of edges of $G^0$ in the $F_-$-neighbourhood of $u$.  Respectively $e_\pm(u):=e(G^0[N_{F_-}(u),N_{F_+}(u)])$ and $e_+(u):=e(G^0[N_{F_+}(u)])$.    Note that $e(u)=e_-(u)+e_{\pm}(u)+e_+(u)$.  Using this notation we see that for a fixed graph $G^0$ satisfying $E_0$, we have
\begin{align*}
\mathbb{E}\left[\triangle_{(1,2)}|E(\alpha),G^0\right]\, &=\, r_1^2\, \sum_{u} \, \Big((1+\alpha)^2 e_-(u)\, +\, (1+\alpha)(1-\alpha)e_{\pm}(u)\, +\, (1-\alpha)^2 e_{-}(u)\Big)\\
&\hspace{-24mm} =\, r_1^2 \sum_{u} e(u)\, +\, 2r_1^2 \alpha\sum_{u} \big(e_-(u)-e_+(u)\big)\, +\, r_1^2\alpha^2 \sum_u \big(e_-(u)+e_+(u)-e_{\pm}(u)\big)\, ,
\end{align*}
up to an $\pm O(n)$ error on each line.  Rather than conditioning on a fixed graph $G^0$ we need to condition just that $G^0$ is in $E_0$.  Doing so corresponds to taking the conditional expected value.  Let $\mu_-=\Ex{e_-(u)|E_0}$, $\mu_+=\Ex{e_+(u)|E_0}$ and $\mu_{\pm}=\Ex{e_{\pm}(u)|E_0}$, noting of course that these values are the same for all vertices $u$.  And so
\eq{was12}
\mathbb{E}[\triangle_{(1,2)}|E_\alpha]\, =\, 3p_0p_1^2\binom{n}{3}\, +\, 2r_1^2\alpha \big(\mu_-\, -\, \mu_{+}\big)n\, +\, r_1^2\alpha^2\big(\mu_-\, +\, \mu_{+}\, -\, \mu_{\pm}\big)n\, \pm \, O(n)\, .
\eqe
In what follows, in the $O(\cdot)$ notation the implicit constant may only depend on the constant $\lambda$ of Proposition~\ref{prop:main} and is not allowed to depend on other quantities such as our choices of $\eta$ or $C'$.  In this context, we can state that $r_1=O(\eta)$.

As $r_1=O(\eta)$ and $\alpha\le n^{-1/4}$ (see~\eqr{alphale}), we can see from~\eqr{was12} that it will suffice to prove that $\mu_- -\mu_+=O(n^{3/2})$ and $\mu_- + \mu_{+} - \mu_{\pm}=O(n^{8/5})$.  In fact, we can do much better than that: we prove that for \emph{every} $G^0$ satisfying $E_0$ we have
\eq{mmp}
\sum_{u\in V(G^0)}\big(e_-(u)\, -\, e_+(u)\big)\, =\, O(n^{5/2})\, ,
\eqe
and
\eq{mpp}
e_-(u)\, +\, e_+(u)\, -\, e_{\pm}(u)\, =\, O(n^{8/5})\phantom{\Big|}
\eqe
for all vertices $u\in V(G^0)$.  Doing so will complete the proof of Lemma~\ref{lem:1,2}.

Let us first introduce some notation that encapsulates the idea of \textit{conditional deviation}. Let  $\bar{d}(u):=(n-d(u)-1)/2$, half the number of non-edges of $u$ in $G^0$.  Let  $D_+(u):=d_{F_+}(u)-\bar{d}(u)$ and $D_-(u):=d_{F_-}(u)-\bar{d}(u)$ be the deviations of these degrees.  Using the handshaking lemma twice, we see that
\[
\sum_{u\in V(G^0)} D_+(u)\, =\, 2e(F_+)\, -\, (N-m_0)\, =\, 0 \, .
\]
Similarly $\sum_{u}D_-(u)=0$.


\textbf{Proof of~\eqr{mmp}}  Let $G^0$ be a graph satisfying $E_0$, that is, satisfying~\eqr{P1}-~\eqr{P6}.  By~\eqr{F-degs} (which is a consequence of~\eqr{P2}) we know that $|D_{+}(u)|,|D_{-}(u)|\le 2\eps n$, for all $u\in V(G^0)$.  The other source of ``error'' in $e_-(u)$ and $e_+(u)$ comes from that deviation of the edge count of $G^0$ in the respective neighbourhoods $N_{F-}(u)$ and $N_{F+}(u)$.  These deviations are bounded by $n^{3/2}$ by~\eqr{P5}.  For example, as $d_{F-}(u)=\bar{d}(u)+D_-(u)$, we have
\begin{align}
    e_-(u)\, & =\,  e(G^0[N_{F-}(u)])\nonumber\\
    & =\, p_0\binom{d_{F_-}(u)}{2}\, \pm\, n^{3/2}\nonumber \\
    & =\, \frac{p_0\big(\bar{d}(u)+D_-(u)\big)^2}{2}\, \pm \, O(n^{3/2})\nonumber \\
    & =\, \frac{p_0\bar{d}(u)^2}{2}\, +\, p_0\bar{d}(u)D_-(u)\, +\, \frac{p_0D_-(u)^2}{2}\, \pm \, O(n^{3/2})\, .\label{eq:e-}
    \end{align}
    By a similar argument
    \eq{e+}
e_+(u)\, =\, \frac{p_0\bar{d}(u)^2}{2}\, +\, p_0\bar{d}(u)D_+(u)\, +\, \frac{p_0D_+(u)^2}{2}\, \pm \, O(n^{3/2})\, .
    \eqe
As $D_+(u)=-D_-(u)$, we have that $D_{+}(u)^2=D_{-}(u)^{2}$, and so these terms cancel, and we see that
\[
e_-(u)\, -\, e_+(u)\, =\, p_0\bar{d}(u)\big(D_-(u)-D_+(u)\big)\, \pm \, O(n^{3/2})\, =\, 2p_0\bar{d}(u)D_-(u)\, \pm \, O(n^{3/2})\, .
\]
We now sum this expression over $u\in V(G^0)$.  We shall write $\bar{d}(u)=(n-d_{G^0}(u)-1)/2=(1-p_0)(n-1)/2-D_{G^0}(u)/2$.  We obtain
\begin{align*}
\sum_{u\in V(G^0)}\big(e_-(u)\, -\, e_+(u)\big)\, &=\, 2p_0\sum_u \bar{d}(u)D_-(u)\, \pm \, O(n^{5/2})\\
&\, =\, p_0(1-p_0)n\sum_u D_-(u)\, -\, p_0\sum_u D_{G^0}(u) D_-(u)\, \pm \, O(n^{5/2})\\
& =\, O(n^{5/2})\, .\phantom{\big|}
\end{align*}
The final bound uses that $\sum_{u}D_-(u)=0$ and the estimates $|D_{G^0}(u)|\le n^{1/2}\log{n}$ by~\eqr{P3} and $|D_-(u)|\le 2\eps n$, which we stated at the start of the proof of~\eqr{mmp}.

\textbf{Proof of~\eqr{mpp}}  We begin by considering $e_{\pm}(u)$.  Using~\eqr{P5}, and the fact that $D_+(u)=-D_-(u)$, we have
\begin{align}
    e_{\pm}(u)\, &=\, e(G^0[N_{F-}(u),N_{F_{+}}(u)])\nonumber \phantom{\big|}\\ &=\, p_0d_{F_-}(u)d_{F_+}(u)\, \pm\, n^{3/2}\phantom{\Big|}\nonumber \\
    & =\, p_0\big(\bar{d}(u)+D_-(u)\big)\big(\bar{d}(u)+D_+(u)\big)\, \pm \, n^{3/2}\phantom{\Big|}\nonumber \\
    & =\, p_0\bar{d}(u)^2\,  -\, p_0D_-(u)^2 \, \pm \, n^{3/2}\, .\phantom{\big|}\label{eq:e+-}
    \end{align}
From~\eqr{e-},~\eqr{e+} and~\eqr{e+-}, we see that
\[
e_-(u)\, +\, e_+(u)\, -\, e_{\pm}(u)\, =\, 2p_0D_{-}(u)^2\, \pm\, O(n^{3/2}).
\]
We now again use that $|D_{-}(u)|\le 2\eps n$, and $\eps=n^{-1/5}$.  It follows immediately that $e_-(u)\, +\, e_+(u)\, -\, e_{\pm}(u)=O(n^{8/5})$, as required.

  \subsection{Proof of Lemma \ref{lem:0,3}}
We now proceed to the final type of triangle, $\triangle_{(0,3)}$.  A triangle of type $(0,3)$ consists of a set of three vertices $u,v,w$, such that all three pairs $uv,uw,vw$ appear as edges of $G^1$.  So that
\[
\triangle_{(0,3)}\, =\, \sum_{\{u,v,w\}}1_{uv,uw,vw \in E(G^1)}\, .
\]
Each edge of $E(K_n)\setminus E(G^0)$ is in $G^1$ with probability $r_1:=p_1/(1-p_0)$, and each such trio of edges is in $G^1$ with probability $r_1^3 +O(n^{-2})$.  It follows that
\[
\mathbb{E}[\triangle_{(0,3)}|G^0]\, =\, r_1^3 N_{\triangle}((G^{0})^c)\, \pm\, O(n)\, .
\]
If we condition only on the event $E_0$  then, as $r_1(1-p_0)=p_1$, and using Lemma~\ref{lem:E0 good}, we have
\begin{align*}
\Ex{\triangle_{(0,3)}|E_0}\, &=\, r_1^3 \Ex{N_{\triangle}((G^{0})^c)|E_0}\, \pm\, O(n)\\
&=\,  r_1^3  \Ex{N_{\triangle}((G^{0})^c)}\, \pm\, O(n)\\ 
&=\, p_1^3 \binom{n}{3}\, \pm\, O(n)\, .
\end{align*}
What changes when we condition on $E_{\alpha}$?  As $E_{\alpha}=E_0\cap E(\alpha)$, for a graph $G^0$ satisfying the conditions of $E_0$, the conditioning affects the probabilities that the edges $uv,uw$ and $vw$ are included in $G^1$.  As they will now depend on whether these pairs belong to $F_-$ or $F_+$.  If all three belong to $F_-$ then the probability $r_1^2 +O(n^{-2})$ is multiplied by a factor of $(1+\alpha)^3$.  More generally, the factor is $(1+\alpha)^2(1-\alpha)$ if just two are in $F_-$, it is $(1+\alpha)(1-\alpha)^2$ if just one is in $F_-$, and is $(1-\alpha)^3$ if all three are in $F_+$.  For $i\in \{0,1,2,3\}$,  let $T_i$ denote the number of triangles in $(G^{0})^c$, which have precisely $i$ edges in $F_-$ and $3-i$ edges in $F_+$.  

By the above discussion, we see that
\begin{equation}\label{eq:eq-0,3}
			\begin{split}
\mathbb{E}\left[\triangle_{(0,3)}|E(\alpha), G^0\right]\, &=\,  p_1^3\binom{n}{3}\, \pm\, O(n)\\
    & +\alpha r_1^3 \big(3T_3\, +\, T_2\, -\, T_1\, -\, 3T_0\big)\\
& +\, \alpha^2 r_1^3\big(3T_3\, -\, T_2\, -\, T_1\, +\, 3T_0\big)\\
& +\, \alpha^3 r_1^3\big(T_3-T_2+T_1-T_0\big)\, .
			\end{split}
		\end{equation}

We shall prove for every choice of $G^0$ which satisfies $E_0$ (that is, satisfies~\eqr{P1}-~\eqr{P6}) we have
\eq{a}
3T_3\, +\, T_2\, -\, T_1\, -\, 3T_0\, =\, O(n^{5/2})
\eqe
and
\eq{a2}
3T_3\, -\, T_2\, -\, T_1\, +\, 3T_0\, =\, \frac{(1-p_0)^3 n^3}{2}\, -\, 3N_{\triangle}((G^{0})^c)\, \pm\, O(n^{13/5})\, .
\eqe
As in the previous subsection the implicit constants in the $O(\cdot)$ notation may only depend on the constant $\lambda$ and not on $\eta$ or $C'$.

Let us now observe that Lemma~\ref{lem:0,3} would follow.  Recall that $r_1=O(\eta)$.  If~\eqr{a} holds then the
second line of~\eqr{eq-0,3} is of the form $O(\alpha \eta^3 n^{5/2})$.  Note also that the final line is trivially at most $O(\alpha^3 \eta^3 n^3)$, which is also $O(\alpha \eta^3 n^{5/2})$, as $\alpha\le n^{-1/4}$ by~\eqr{alphale}.    It follows that, if $G^0$ satisfies~\eqr{a} and~\eqr{a2}, then
\[
\mathbb{E}\left[\triangle_{(0,3)}|E(\alpha), G^0\right]\, =\,  p_1^3\binom{n}{3}\, \pm \, O\left(\alpha^2 \eta^3 \left[\frac{(1-p_0)^3 n^3}{2}\, -\, 3N_{\triangle}((G^{0})^c)\right]\right)\, \pm\, O(\alpha \eta^3 n^{5/2})\, .
\]
The conditional expectation given the event $E_{\alpha}$ is obtained by averaging this over $G^0$ satisfying $E_0$, as we commented above in~\eqr{conditionE}. By Lemma~\ref{lem:E0 good}, and the fact that $\Ex{N_{\triangle}((G^{0})^c)}=(1-p_0)^3\binom{n}{3}=(1-p_0)^3 n^3/6\pm O(n^{2})$, the first error term, once averaged over $G^0$ satisfying $E_0$, is of the form $O(\alpha^2\eta^3 n^2)$.  We obtain that
\[
\mathbb{E}\left[\triangle_{(0,3)}|E_{\alpha}\right]\, =\,  p_1^3\binom{n}{3}\, \pm\, O(\alpha \eta^3 n^{5/2})\, ,
\]
as required.  Thus, to prove Lemma~\ref{lem:0,3} it remains only to prove that~\eqr{a} and~\eqr{a2} hold for all graphs $G^0$ in $E_0$ (that is, satisfying~\eqr{P1}-~\eqr{P6}).

\textbf{Proof of~\eqr{a}} We begin by noting that $3T_3\, +\, T_2\, -\, T_1\, -\, 3T_0$ is the same value as that obtained by summing over triangles of $(G^{0})^c$ the number of edges in $F_-$ minus the number of edges in $F_+$.  Let $K_3(G)$ denote the set of triangles (a triangle being considered a set of three edges) in a graph $G$.  We have
\begin{equation*}
\begin{split}
3T_3\, +\, T_2\, -\, T_1\, -\, 3T_0\,&=\, \sum_{A\subset K_3((G^{0})^c)}|A\cap F_-|-|A\cap F_+|\\
&=\,\sum_{uw\not\in G^0}d_{(G^{0})^c}(u,w)\big(1_{uw\in F_-}\, -\, 1_{uw\in F_+}\big)\\
&=\,\sum_{uw\in F_-}D_{(G^{0})^c}(u,w)\, -\, \sum_{uw\in F_+}D_{(G^{0})^c}(u,w)\\
&\leq\, \left((N-m_0)\sum_{uw\in (G^{0})^c}D_{(G^{0})^c}(u,w)^2\right)^{1/2}\,\\
&\leq C_D^{1/2}n^{5/2},\phantom{\big|}
\end{split}
\end{equation*}
where we used that $|F_-|=|F_+|$ for the third equality, the Cauchy-Schwarz inequality for the first inequality and the fact that $G^0$ satisfies \eqr{P1} for the final inequality.

\textbf{Proof of~\eqr{a2}}

We may relate $3T_3\, -\, T_2\, -\, T_1\, +\, 3T_0$ to the number of monochromatic triangles in a two colouring of the edges of $(G^{0})^c$.  Indeed, if edges of $F_-$ and $F_+$ are coloured red and blue respectively, and $N_{MC}$ denotes the number of monochromatic triangles, then we have
\eq{tomono}
3T_3\, -\, T_2\, -\, T_1\, +\, 3T_0\, =\, 4N_{MC}\, -\, N_{\triangle}((G^{0})^c)\, .
\eqe
We now state a Goodman type result for monocromatic triangles in an arbitrary coloured graph.  The proof is essentially identical to that of Goodman's formula, but we include it for completeness.

\begin{lem}\label{lem:goodman}
    Let $G$ be a graph and suppose its edges are red/blue coloured. Let $N_{MC}$ be the number of monocromatic triangles in this colouring and let $N_R(v)$ and $N_B(v)$ denote the red and blue neighbourhoods of $v$, respectively. Then
    $$2N_{MC}\,=\,\sum_{v\in V(G)}e\big(G[N_R(v)]\big)+\sum_{v\in V(G)}e\big(G[N_B(v)]\big)\, -\, N_\triangle(G).$$
\end{lem}

\begin{proof}
    Let's count how much each type of triangle contributes to the RHS. A red triangle contributes $3$ to the first sum, $0$ to the second and $-1$ to the third one, giving a total of $2$. Similarly, a blue triangle contributes $0$, $3$ and $-1$, giving a total of $2$. Now, let $uvw$ be a non-monocramatic triangle in $G$.  Without loss of generality suppose that $uv$ and $uw$ are red and that $vw$ is blue.  Then, this triangle adds $1$ (as $vw$ is an edge of $G[N_R(u)]$) to the first sum, and $0$ and $-1$ to the other two sums.  This gives a total contribution of $0$. Therefore, the only triangles contributing to the sum are monocromatic and each one is counted twice, as claimed.
\end{proof}

From the lemma and~\eqr{tomono}, we see that $3T_3\, -\, T_2\, -\, T_1\, +\, 3T_0$ is precisely
\[
2\sum_{u}e\big((G^{0})^c[N_{F_-}(u)]\big)\, 
+\, 2\sum_{u}e\big((G^{0})^c[N_{F+}(u)]\big)\, -\, 3N_{\triangle}((G^{0})^c)\, .
\]
We now use~\eqr{P5} to control the number of edges in the neighbourhoods, and that $2\binom{d}{2}=d^2+O(n)$ for any $d\le n$.  Doing so, and using the notation $q_0:=1-p_0$,  we see that $3T_3\, -\, T_2\, -\, T_1\, +\, 3T_0$ is given by
\[
q_0 \sum_{u}d_{F_-}(u)^2\, 
+\, q_0 \sum_{u} d_{F+}(u)^2\, -\, 3N_{\triangle}((G^{0})^c)\, \pm \, O(n^{5/2}).
\]
We may now write $d_{F_-}(u)=\bar{d}(u)+D_{-}(u)$, as in the previous subsection.  Our expression for $3T_3\, -\, T_2\, -\, T_1\, +\, 3T_0$ becomes
\[
q_0 \sum_{u}\big(\bar{d}(u)+D_{-}(u)\big)^2\, 
+\, q_0 \sum_{u}\big(\bar{d}(u)+D_{+}(u)\big)^2\, -\, 3N_{\triangle}((G^{0})^c)\, \pm \, O(n^{5/2})\, .
\]
As $D_+(u)=-D_-(u)$, the linear term will cancel, leaving
\begin{align*}
q_0 &\sum_{u}\big(\bar{d}(u)^2 \, +\, D_{-}(u)^2\big)\, 
+\, q_0 \sum_{u}\big(\bar{d}(u)^2 \, +\, D_{+}(u)^2\big)\, -\, 3N_{\triangle}((G^{0})^c)\, \pm \, O(n^{5/2})\\
& =\, 2q_0\sum_{u}\bar{d}(u)^2\, +\, 2q_0\sum_{u}D_-(u)^2\, -\, 3N_{\triangle}((G^{0})^c)\, \pm \, O(n^{5/2}).
\end{align*}

We shall now write $\bar{d}(u)=(n-d_{G^0}(u)-1)/2=q_0(n-1)/2-D_{G^0}(u)/2$.  As the linear term in $D_{G^0}(u)$ cancels, our expression for $3T_3\, -\, T_2\, -\, T_1\, +\, 3T_0$ becomes
\[
    \frac{q_0^3 n^3}{2}\, +\, \frac{q_0}{2}\sum_u D_{G^0}(u)^2\, +\, 2q_0\sum_{u}D_-(u)^2\, -\, 3N_{\triangle}((G^{0})^c)\, \pm \, O(n^{5/2})\, .
\]
Finally, using that $|D_{-}(u)|\le 2\eps n=O(n^{4/5})$ due to~\eqr{P2} (see~\eqr{F-degs}) and~\eqr{P3} to control to control $|D_{G^0}(u)|$, we see that
\[
3T_3\, -\, T_2\, -\, T_1\, +\, 3T_0\, =\, \frac{q_0^3 n^3}{2}\, -\, 3N_{\triangle}((G^{0})^c)\, \pm \, O(n^{13/5})\, ,
\]
which completes the proof of~\eqr{a2}.\phantom{--------------------}\qed

\section{Concluding Remarks}

The main open problem we would like to highlight is Conjecture~\ref{conj:cont}.  We are relatively confident that, based on the ``continuity'' that we showed for the lower tail of the triangle count, the same ought to be true for the count of graphs $H$, in the case where $H$ contains a triangle.

The assertion that a discontinuity does occur for other graphs $H$ is more speculative, and is partly inspired by Siderenko's conjecture for bipartite graphs.  We would be very interested to see more evidence for (or against) our conjecture.

It would also be of interest to find the correct constant in Theorem~\ref{thm:main}.  That is, determine the limit
\[
\lim_{n\to \infty} \frac{-\log{\pr{N_{\triangle}(G_{(m)})<(1-\delta)\mu_{n,m}}}}{\delta^2n^3}\phantom{\Bigg|}\, ,
\]
assuming it exists, where $m=p\binom{n}{2}$ and $n^{-1}\ll \delta\ll n^{-3/4}$.

Let us now comment on the relation between our results and those of Neeman, Radin and Sadun~\cite{NRS22}.  They showed, in the regime $n^{-3/4}\ll\delta\ll 1$ that the easiest (i.e., most likely) way for the event $N_{\triangle}(G_{(m)})<(1-\delta)\mu_{n,m}$ to occur is for the smallest eigenvalue is be around $-\delta^{1/3}pn$, see Theorem 1 of~\cite{NRS22}.  Note that this eigenvalue ``causes'' the fall in the triangle count, as the triangle count is one sixth of the sum of the cubes of the eigenvalues of the adjacency matrix (as can be seen by considering the trace of $A^3$).

It is therefore natural to ask whether our lower bound on the probability of the event $N_{\triangle}(G_{(m)})<(1-\delta)\mu_{n,m}$, in the regime $n^{-1}\ll \delta\ll n^{-3/4}$, is also related to the distribution of eigenvalues of the adjacency matrix.  We believe that there is a connection, albeit somewhat more subtle.  

In the proof we generate $G$ as $G^0\cup G^1$, and consider the event $E(\alpha)$ which biases us towards selecting more edges of $G^1$ to be from $F_-$ (which are pairs with smaller synergy).  The distribution of pairs in $F_-$ is not completely uniform.  Indeed, the reader may convince themselves that there are likely to be more pairs of $F_-$ across a bipartition $(U,W)$ if that bipartition is \emph{either} significantly more dense \emph{or} significantly less dense than the general density of the graph.

As large positive/negative eigenvalues correspond to less/more dense bipartite structures, conditioning on $E(\alpha)$ has a tendency to reduce the value of both extremes.  As $\sum_{i=1}^n\lambda_n=0$, we would then expect a slight increase in the remaining eigenvalues.  In this way, it seems that conditioning on $E(\alpha)$ tends to reduce $\sum_{i}\lambda_i^3$ by skewing of the distribution of eigenvalues.  This is quite different from an approach based on one, or just a few, eigenvalues.

Finally, we would be very interested to see other applications of Propositions~\ref{prop:syngm} and ~\ref{prop:syngp}.  These propositions, which states that the vector of synergies is normally distributed with very high probability, may also be useful for other problems related to subgraphs of random graphs.

\bibliographystyle{plain}
\bibliography{template}

\newpage

\appendix

\section{Standard Deviations}
During Section \ref{sec:syn} we defined the following parameter
$$\sigma(p)^2\, :=\, \Var\big(S_{G_p}^p(u,w)\mid uw\not\in E(G_p)\big).$$
We deduce its value in the following claim.
\begin{claim}\label{claim:sigmap}
    $$\sigma(p)^2=p^2(1-p)^2(n-2).$$
\end{claim}
\begin{proof}
    First, recall that $\Ex{S_{G_p}^p(u,w)\mid uw\not\in E(G_p)}=0$ and hence $\sigma(p)^2=\mathbb{E}[S_{G_p}^p(u,w)^2\mid uw\not\in E(G_p)]$. As $S_{G_p}^p(u,w)$ can be written as $\sum_{v\in V\setminus\{u,w\}}Y_v$ where $Y_v$ are iid random variables which take  values
    
    \begin{center}
        $\begin{matrix}
       (1-p)^2\qquad & \text{with probability}\qquad & p^2 \\
        p(p-1)\qquad &  & 2p(1-p) \\
        p^2 \qquad& & (1-p)^2.
    \end{matrix}$
    \end{center}
    Using independence,
    \begin{equation*}
        \begin{split}
        \mathbb{E}[S_{G_p}^p(u,w)^2]\,=\,\mathbb{E}\left[\left(\sum_{v\in V\setminus\{u,w\}}Y_v\right)^2\right]\,&=\,\sum_{v\in V\setminus\{u,w\}}\mathbb{E}\big[Y_v^2]\\
        &=\, p^2(1-p)^2(n-2),\phantom{big|}
        \end{split}
    \end{equation*}
    as claimed.
\end{proof}

\section{Hypergeometrics}

The following lemma gives a lower bound for the tail of the hypergeometric distribution in the case that exactly half of the population count as ``successes''.  We remark that now that we are in the appendix we use $p$ for the density of a random set (and not $m/N$, as in the rest of the article).

\begin{lem}\label{lem:hyper} Let $p,\alpha\in (0,1)$, with $(1+\alpha)p\le 1$ and let $M\in \mathbb{N}$ be such that $(1+\alpha)pM$ and $2pM$ are integers.  Let $A$ be a uniformly random subset of $2pM$ elements of $[2M]$.  Then
\[
\pr{|A\cap [M]|=(1+\alpha)pM}\, \ge\, \exp\left(\frac{-p }{1-p} \alpha^2 M\, -\, \frac{p-2p^2+2p^3}{(1-p)^2}\alpha^3 M\, \pm\, O(\log{M})\right)\, .
\]
\end{lem}

\begin{remark} In particular, if $\alpha^3 M\ge \log{M}$, then this bound is of the form $\exp(-p\alpha^2M/(1-p)\, +\, O_p(\alpha^3M))$.
\end{remark}

\begin{proof} We begin by noting that the probability in question is simply
\[
\pr{|A\cap [M]|=(1+\alpha)pM}\, =\, \frac{\binom{M}{(1+\alpha)pM}\binom{M}{(1-\alpha)pM}}{\binom{2M}{2pM}}\, .
\]
We now use Stirling's approximation $n!\sim \sqrt{2\pi n}(n/e)^n$, which gives, in this context, with $a\in(0,1)$, that
\[
\binom{n}{an}\, =\, n^{O(1)}\frac{n^n}{(an)^{an}(n(1-a))^{n(1-a)}}\, =\, \frac{n^{O(1)}}{a^{an}(1-a)^{(1-a)n}} \, . \phantom{\Bigg|}
\]

And so the probability is given by
\begin{align*}
\frac{M^{O(1)}p^{2pM}(1-p)^{2(1-p)M}}{((1+\alpha)p)^{(1+\alpha)pM}((1-\alpha)p)^{(1-\alpha)pM}(1-p-p\alpha)^{(1-p-p\alpha)M}(1-p+p\alpha)^{(1-p+p\alpha)M}}
\end{align*}
We note that the $p$ terms cancel.  If we also divide top and bottom by $(1-p)^{2(1-p)M}$ we obtain
\[
M^{O(1)}(1+\alpha)^{-(1+\alpha)pM}(1-\alpha)^{-(1-\alpha)pM}\left(1-\frac{p\alpha}{(1-p)}\right)^{-(1-p-p\alpha)M}\left(1+\frac{p\alpha}{(1-p)}\right)^{-(1-p+p\alpha)M}
\]
which, by the difference of two squares $(1-\alpha^2)=(1-\alpha)(1+\alpha)$, becomes 
\[
M^{O(1)}(1-\alpha^2)^{-(1-\alpha)pM}(1+\alpha)^{-2\alpha pM}\left(1-\frac{p^2\alpha^2}{(1-p)^2}\right)^{-(1-p-p\alpha)M}\left(1+\frac{p\alpha}{(1-p)}\right)^{-2\alpha pM}\, .
\]
As $1+\alpha\le e^{\alpha}$ and $1-\alpha\le e^{-\alpha}$, it follows that the probability is at least
\[
\exp\left(\alpha^2 M\left[p-2p+\frac{p^2}{(1-p)}-\frac{2p^2}{(1-p)}\right]\, -\, \alpha^3 M\left[p+\frac{p^3}{(1-p)^2}\right]\, \pm\, O(\log{M})\right).
\]
This simplifies to
\[
\exp\left(-\alpha^2 M\left[p+\frac{p^2}{(1-p)}\right]\, -\, \alpha^3 M\left[p+\frac{p^3}{(1-p)^2}\right]\, \pm\, O(\log{M})\right)
\]
which is equal to the desired expression.
\end{proof}

We shall only require this much weaker result, which we state as a corollary.

\begin{cor}\label{cor:hyper}
Given $\lambda>0$, for all $p\le 1-\lambda$ the following holds.  Let $X$ be distributed as the hypergeometric random variable $\textrm{Hyper}(2M,M,2pM)$ and let $M^{-1/2}\log{M}\ll \alpha\ll 1$. Then, for all sufficiently large $M$, we have
\[
\pr{X=\lfloor(1+\alpha)pM\rfloor}\, \ge\, \exp(-\lambda^{-1}\alpha^2 M)\, .
\]
\end{cor}  

\begin{proof}
The random variable $X$ is distributed as $|A\cap [M]|$, the random variable considered in Lemma~\ref{lem:hyper}.  It follows that
\[
\pr{X=\lfloor(1+\alpha)pM\rfloor}\, \ge\, \exp\left(\frac{-p }{1-p} \alpha^2 M\, -\, \frac{p-2p^2+2p^3}{(1-p)^2}\alpha^3 M\, \pm\, O(\log{M})\right)\, .
\]
By the conditions on $\alpha$ both the second and third terms are $o(\alpha^2M)$.  As the $p/(1-p)< \lambda^{-1}$, the claimed bound holds for all sufficiently large $M$.
\end{proof}

We remark that an alternative proof of Lemma~\ref{lem:hyper}, that would be even more precise, could be obtained by comparison with the conditioned Binomial distribution.  Let $A_p$ be the binomial random subset of $[2M]$ in which each element is included independently with probability $p$.  One may observe that
\[
\pr{|A\cap [M]|=(1+\alpha)pM}\, =\, \pr{|A_p\cap [M]|=(1+\alpha)pM\, \mid\, |A_p|=2pM}\, .
\]
And so this probability is given by
\begin{align*}
&\frac{\pr{|A_p\cap [M]|=(1+\alpha)pM, |A_p|=2pM}}{\pr{|A_p|=2pM}}\\
&=\, \pr{|A_p|=2pM\, \mid\, |A_p\cap [M]|=(1+\alpha)pM}\frac{\pr{|A_p\cap [M]|=(1+\alpha)pM}}{\pr{|A_p|=2pM}}\phantom{Bigg|}\\
& =\, M^{O(1)}\, \pr{|A_p\cap [M]|=(1-\alpha)pM}\pr{|A_p\cap [M]|=(1+\alpha)pM}\, .\phantom{Bigg|}
\end{align*}
One may then understand this probability in terms of the known tail bounds on the binomial distribution, such as that stated as Theorem 1.13 in~\cite{GGS}, which was adapated from~\cite{Bah}.

\end{document}